\documentclass[twoside]{article}

\usepackage[accepted]{aistats2021}

\usepackage[utf8]{inputenc} 
\usepackage[T1]{fontenc}    
\usepackage{hyperref}       
\usepackage{url}            
\usepackage{booktabs}       
\usepackage{amsfonts}       
\usepackage{nicefrac}       
\usepackage{microtype}      
\usepackage{amsmath}
\usepackage{amsthm}
\usepackage{mathrsfs}
\usepackage{amssymb}
\usepackage{graphicx}
\usepackage{color}
\usepackage[ruled,vlined,linesnumbered]{algorithm2e}
\usepackage{multirow}
\usepackage{enumitem}
\usepackage{xcolor}
\usepackage{caption} 
\usepackage{subcaption} 

\usepackage[round]{natbib}
\usepackage[section]{placeins}

\usepackage{mathtools}

\usepackage[normalem]{ulem} 

\newcommand{\va}{{\mathbf{a}}}
\newcommand{\vb}{{\mathbf{b}}}
\newcommand{\vc}{{\mathbf{c}}}
\newcommand{\vd}{{\mathbf{d}}}

\newcommand{\vf}{{\mathbf{f}}}

\newcommand{\vp}{{\mathbf{p}}}

\newcommand{\vs}{{\mathbf{s}}}

\newcommand{\vu}{{\mathbf{u}}}
\newcommand{\vv}{{\mathbf{v}}}

\newcommand{\vx}{{\mathbf{x}}}
\newcommand{\vy}{{\mathbf{y}}}
\newcommand{\vz}{{\mathbf{z}}}

\newcommand{\vA}{{\mathbf{A}}}
\newcommand{\vB}{{\mathbf{B}}}

\newcommand{\vD}{{\mathbf{D}}}

\newcommand{\vI}{{\mathbf{I}}}

\newcommand{\vQ}{{\mathbf{Q}}}

\newcommand{\vV}{{\mathbf{V}}}

\newcommand{\vX}{{\mathbf{X}}}

\newcommand{\cL}{{\mathcal{L}}}

\newcommand{\cN}{{\mathcal{N}}}

\newcommand{\cX}{{\mathcal{X}}}

\newcommand{\vareps}{\varepsilon}

\newcommand{\RR}{\mathbb{R}} 
\newcommand{\ZZ}{\mathbb{Z}} 
\newcommand{\vzero}{\mathbf{0}} 
\newcommand{\vone}{{\mathbf{1}}} 

\newcommand{\dist}{\mathrm{dist}}    
\newcommand{\dom}{{\mathrm{dom}}} 
\newcommand{\Proj}{{\mathrm{Proj}}} 

\newcommand{\st}{\mathrm{~s.t.~}}

\DeclareMathOperator*{\argmin}{arg\,min} 

\DeclareMathOperator*{\Min}{minimize}

\newcommand{\bc}{\begin{center}}
\newcommand{\ec}{\end{center}}

\newcommand{\bdm}{\begin{displaymath}}
\newcommand{\edm}{\end{displaymath}}

\newcommand{\beq}{\begin{equation}}
\newcommand{\eeq}{\end{equation}}

\newcommand{\bfl}{\begin{flushleft}}
\newcommand{\efl}{\end{flushleft}}

\newcommand{\bt}{\begin{tabbing}}
\newcommand{\et}{\end{tabbing}}

\newcommand{\beqn}{\begin{eqnarray}}
\newcommand{\eeqn}{\end{eqnarray}}

\newcommand{\beqs}{\begin{align*}} 
\newcommand{\eeqs}{\end{align*}}  

\newtheorem{theorem}{Theorem}
\newtheorem{claim}{Claim}
\newtheorem{definition}{Definition}
\newtheorem{lemma}{Lemma}
\newtheorem{remark}{Remark}
\newtheorem{assumption}{Assumption}

\newcommand{\SL}[1]{\textcolor{orange}{SL: #1}}

\begin{document}

\twocolumn[

\aistatstitle{Rate-improved Inexact Augmented Lagrangian Method for Constrained Nonconvex Optimization}

\aistatsauthor{ Zichong Li$^1$ \And Pin-Yu Chen$^{*,2}$ \And  Sijia Liu$^{*,3}$ \And
 Songtao Lu$^{*,2}$ \And Yangyang Xu$^{*,1}$ 
}
\aistatsaddress{$^1$Rensselaer Polytechnic Institute \And  $^2$IBM Research \And $^3$Michigan State University}
]

\begin{abstract}
First-order methods have been studied for nonlinear constrained optimization within the framework of the augmented Lagrangian method (ALM) or penalty method. We propose an improved inexact ALM (iALM) and conduct a unified analysis for nonconvex problems with either affine equality or nonconvex constraints. Under certain regularity conditions (that are also assumed by existing works), we show an $\tilde{O}(\varepsilon^{-\frac{5}{2}})$ complexity result for a problem with a nonconvex objective and affine equality constraints and an $\tilde{O}(\varepsilon^{-3})$ complexity result for a problem with a nonconvex objective and nonconvex constraints, where the complexity is measured by the number of first-order oracles to yield an $\varepsilon$-KKT solution. Both results are the best known. The same-order complexity results have been achieved by penalty methods. However, two different analysis techniques are used to obtain the results, and more importantly, the penalty methods generally perform significantly worse than iALM in practice. Our improved iALM and analysis close the gap between theory and practice. Numerical experiments on nonconvex problems with affine equality or nonconvex constraints are provided to demonstrate the effectiveness of our proposed method.
\end{abstract}

\section{INTRODUCTION}

First-order methods (FOMs) have been extensively used for solving large-scale optimization problems, partly due to its nice scalability. Compared to second-order or higher-order methods, FOMs generally have much lower per-iteration complexity and much lower requirement on machine memory. A majority of existing works on FOMs focus on problems without constraints or with simple constraints, e.g., \citep{nesterov2013gradient, FISTA2009, ghadimi2016accelerated, carmon2018accelerated, snap19}. Several recent works have made efforts on analyzing FOMs for problems with complicated functional constraints, e.g., \citep{yu2017simple, lin2018level, lin2019inexact-pp, xu2019iter-ialm, xu2018pd-sgd-functional, lu2018iteration, li2019-piALM}.

In this paper, we consider \textit{nonconvex} problems with (\textit{possibly nonlinear}) equality constraints, formulated as
\begin{equation}\label{eq:ncp_eq}
f_0^*:=\min_{\vx\in\RR^n} \big\{f_0(\vx) := g(\vx) + h(\vx), \st \vc(\vx)=\vzero \big\},
\end{equation}
where $g$ is continuously differentiable but possibly nonconvex, $h$ is closed convex but possibly nonsmooth, and $\vc=(c_1, \ldots, c_l):\RR^n\to\RR^l$ is a vector function with continuously differentiable components. Note that an inequality constraint $d(\vx)\le 0$ can be equivalently formulated as an equality constraint $d(\vx)+s=0$ by enforcing the nonnegativity of $s$. In addition, the stationary conditions of an inequality-constrained problem and its reformulation can be equivalent, as we will see later at the end of Section~\ref{sec:complexity}. Hence, we do not lose generality by focusing on equality-constrained problems in the form of \eqref{eq:ncp_eq}. A large class of nonlinear constraint problems can be covered by our formulation. Examples include Neyman-Pearson classification with type-I error \citep{neyman1933ix}, resource allocation with nonlinear budgets \citep{katoh1998resource}, and 
quadratically constrained quadratic program.

\subsection{Related works}
The augmented Lagrangian method (ALM) is one of the most popular approaches for solving nonlinear constrained problems. It first appeared in \citep{powell1969method, hestenes1969multiplier}. Based on the augmented Lagrangian (AL) function, ALM alternatingly updates the primal variable by minimizing the AL function and the Lagrangian multiplier by dual gradient ascent. If the multiplier is fixed to zero, then ALM reduces to a standard penalty method. Early works often used second-order methods, such as the Newton's method, to solve primal subproblems of ALM. With the rapid increase of problem size in modern applications and/or existence of non-differentiable terms, second-order methods become extremely expensive or even inapplicable. Recently, more efforts have been made on integrating first-order solvers into the ALM framework and analyzing the AL-based FOMs.

For convex affinely-constrained problems, \citep{lan2016iteration-alm} presents an AL-based FOM that can produce an $\vareps$-KKT point with $O(\vareps^{-1}|\log\vareps|)$ gradient evaluations and matrix-vector multiplications. This result was extended to convex conic programming \citep{lu2018iteration, aybat2013augmented} and to convex nonlinear constrained problems \citep{li2020augmented, li2019-piALM}. When an $\vareps$-optimal solution is desired, $O(\vareps^{-1})$ complexity results have been established for AL-based FOMs in several works, e.g., \citep{xu2019iter-ialm, xu2017first, ouyang2015accelerated, li2019-piALM, nedelcu2014computational}. For strongly-convex problems, the complexity results can be respectively improved to $O(\vareps^{-\frac{1}{2}}|\log\vareps|)$ for an $\vareps$-KKT point and $O(\vareps^{-\frac{1}{2}})$ for an $\vareps$-optimal solution; see \citep{li2020augmented, li2019-piALM, xu2019iter-ialm, nedelcu2014computational, necoara2014rate} for example. 

For nonconvex problems with affine equality constraints, \citep{jiang2019structured} can find an $\vareps$-KKT solution to a similar variant of our problem \eqref{eq:ncp_eq} with $\tilde{O}( \vareps^{-2})$ complexity. Also, \citep{zhang2020global} achieved $O(\vareps^{-2})$ complexity for nonconvex smooth problems with polyhedral constraints. However, both of their analysis heavily exploited the affinity of $c(\cdot)$ and didn't include the case of nonconvex $c(\cdot)$. 

For problems with nonconvex constraints, early works designed and analyzed FOMs in the framework of a penalty method. \citep{cartis2011evaluation} first presents an FOM for minimizing composite functions and then applies it to nonlinear constrained nonconvex optimization within the framework of an exact-penalty method. To obtain an $\vareps$-KKT point, the FOM by \cite{cartis2011evaluation} needs $O(\vareps^{-5})$ gradient evaluations. A follow-up paper by \cite{cartis2014complexity} gives a trust-region based FOM and shows an $O(\vareps^{-2})$ complexity result to produce an $\vareps$-Fritz-John point, which is weaker than an $\vareps$-KKT point. On solving affinely-constrained nonconvex problems, \citep{kong2019complexity-pen} gives a quadratic-penalty-based FOM and establishes an $O(\vareps^{-3})$ complexity result to obtain an $\vareps$-KKT point. When Slater's condition holds, $\tilde O(\vareps^{-\frac{5}{2}})$ complexity results have been shown in \citep{li2020augmented, lin2019inexact-pp}, which consider nonconvex problems with nonlinear convex constraints. While the FOMs in \citep{li2020augmented, lin2019inexact-pp} are penalty-based, the recent work \citep{meloiteration2020} proposes a first-order proximal ALM for affinely-constrained nonconvex problems and obtains an $\tilde{O}(\vareps^{-\frac{5}{2}})$ result. 

Besides AL and penalty-based FOMs, several other FOMs have been designed to solve nonlinear-constrained problems, such as the level-set FOM by \cite{lin2018level} and the primal-dual method by \cite{yu2016primal} for convex problems. FOMs have also been proposed for minimax problems. For example, \citep{hien2017inexact, hamedani2018primal} study FOMs for convex-concave minimax problems, and \citep{nonminmax19,hibsa20,lin2020near} analyzes FOMs for nonconvex-concave minimax problems. While a nonlinear-constrained optimization problem can be formulated as a minimax problem, its KKT conditions are stronger than the stationarity conditions of a nonconvex-concave minimax problem, because the latter with a compact dual domain cannot guarantee primal feasibility. Therefore, stationarity of a minimax problem in \citep{lin2020near} does not imply primal feasibility of our problem.


\begin{table*}[t]\caption{Comparison of the complexity results of several methods in the literature to our method to produce an $\vareps$-KKT solution to \eqref{eq:ncp_eq}. 
}\label{table:comparison} 
\begin{center}
\begin{tabular}{|c|c|c|c|c|c|} 
\hline
Method & type & objective & constraint & regularity & complexity \\\hline\hline
\multirow{2}{*}{iALM~\citep{li2020augmented}} & \multirow{2}{*}{AL} & strongly convex & convex & none & $\tilde{O}(\vareps^{-\frac{1}{2}})$\\ 
 &  &  convex & convex & none & $\tilde{O}(\vareps^{-1})$\\ \hline
QP-AIPP~\citep{kong2019complexity-pen}& penalty & nonconvex & convex & none & $\tilde{O}(\vareps^{-3})$\\\hline
HiAPeM~\citep{li2020augmented} & hybrid & nonconvex & convex & Slater's condition & $\tilde{O}(\vareps^{-\frac{5}{2}})$\\\hline
\multirow{3}{*}{iPPP~\citep{lin2019inexact-pp}} & \multirow{3}{*}{penalty} & \multirow{3}{*}{nonconvex} & convex & Slater's condition & $\tilde{O}(\vareps^{-\frac{5}{2}})$\\ 
& & & nonconvex & none & $\tilde{O}(\vareps^{-4})$\\
& & & nonconvex & Assumption~\ref{assumption:regularity} & $\tilde{O}(\vareps^{-3})$\\\hline
iALM~\citep{sahin2019inexact} & AL & nonconvex & nonconvex & Assumption~\ref{assumption:regularity} & $\tilde{O}(\vareps^{-4})$\\\hline
\multirow{2}{*}{this paper} & \multirow{2}{*}{AL} & \multirow{2}{*}{nonconvex} & convex & Assumption~\ref{assumption:regularity} & $\tilde{O}(\vareps^{-\frac{5}{2}})$\\
& & & nonconvex & Assumption~\ref{assumption:regularity} & $\tilde{O}(\vareps^{-3})$\\\hline
\end{tabular}
\end{center}
\end{table*}

\subsection{Contributions}\label{contributions}
Our contributions are three-fold. First, we propose a novel FOM in the framework of \textit{inexact ALM} (iALM) for nonconvex optimization problems with nonlinear (possibly nonconvex) constraints. Due to nonlinearity and large-scale, it is impossible to exactly solve primal subproblems of ALM, and the iALM instead solves each subproblem approximately to a certain desired accuracy. Different from existing works on iALMs, we use an {inexact proximal point method (iPPM)} to solve each ALM subproblem. The use of iPPM leads to more stable numerical performance and also better theoretical results. Second, we conduct complexity analysis to the proposed iALM. Under a regularity condition, we obtain an $\tilde O(\vareps^{-\frac{5}{2}})$ result if the constraints are convex and an $\tilde O(\vareps^{-3})$ result if the constraints are nonconvex. This yields a substantial improvement over the best known complexity results of AL-based FOMs, $\tilde O(\vareps^{-3})$ \citep{li2020augmented} and $\tilde O(\vareps^{-4})$ \citep{sahin2019inexact} (see Remark \ref{remark:error}) respectively for the aforementioned convex and nonconvex constrained cases. 
While quadratic-penalty-based FOMs (under the same regularity condition as what we assume for nonconvex-constraint problems) \citep{lin2019inexact-pp} have achieved the same-order results as ours, but their empirical performance is generally (much) worse. Hence, our results close the gap between theory and practice. Thirdly, our algorithm and analysis are unified for the convex-constrained and nonconvex-constrained cases. Existing works on penalty-based FOMs such as \citep{lin2019inexact-pp} need different algorithmic designs and also different analysis techniques to obtain the $\tilde O(\vareps^{-\frac{5}{2}})$ and $\tilde O(\vareps^{-3})$ results, separately for the convex-constrained and nonconvex-constrained cases.

\begin{remark}\label{remark:error}
An $\tilde{O}(\vareps^{-3})$ complexity is claimed in Corollary 4.2 in~\citep{sahin2019inexact}. However, this complexity is based on an existing result that was not correctly referred to. The authors claimed that the complexity of solving each nonconvex composite subproblem is $O\left(\frac{\lambda_{\beta_k}^2 \rho^2}{\vareps_{k+1}}\right)$, which should be $O\left(\frac{\lambda_{\beta_k}^2 \rho^2}{\vareps_{k+1}^2}\right)$; see~\citep{sahin2019inexact} for the definitions of $\lambda_{\beta_k}, \rho, \vareps_{k+1}$. Using the correctly referred result and following the same proof in \citep{sahin2019inexact}, we get a total complexity of $\tilde{O}(\vareps^{-4})$.
\end{remark}

\subsection{Complexity comparison on different methods}

In Table \ref{table:comparison}, we summarize our complexity results and several existing ones of first order methods to produce an $\vareps$-KKT solution to \eqref{eq:ncp_eq}. We consider several cases based on whether the objective and the constraints are convex. Here, constraints being convex means that the feasible set is convex, or in other words, equality constraint functions must be affine and inequality constraint functions must be convex. Our result matches the best-known existing results, which are achieved by penalty-type methods such as the iPPP by \cite{lin2019inexact-pp}. In practice, AL-type methods usually significantly outperform penalty-type methods. Hence, our method is competitive in theory and can be significantly better in practice, as we demonstrated in the numerical experiments.

\subsection{Notations, definitions, and assumptions}

We use $\|\cdot\|$ for the Euclidean norm of a vector and the spectral norm of a matrix. For a positive integer, $[n]$ denotes the set $\{1,\ldots,n\}$. The big-$O$ notation is used with standard meaning, while $\tilde O$ suppresses all logarithmic terms of $\vareps$. Given $\vx\in\dom(h)$, we denote $J_c(\vx)$ as the Jacobi matrix of $\vc$ at $\vx$. We denote the distance function between a vector $\vx$ and a set $\cX$ as $\dist(\vx,\cX) = \min_{\vy \in \cX} \Vert \vx-\vy \Vert$.
The augmented Lagrangian (AL) function of \eqref{eq:ncp_eq} is
\begin{equation}\label{eq:al-fun}
\cL_\beta(\vx,\vy) = f_0(\vx) + \vy^\top \vc(\vx) + \frac{\beta}{2}\|\vc(\vx)\|^2,
\end{equation}
where $\beta>0$ is a penalty parameter, and $\vy\in\RR^l$ is the multiplier vector. 

\begin{definition}[$\vareps$-KKT point]\label{def:eps-kkt}
Given $\vareps \geq 0$, a point $\vx \in \RR^n$ is called an $\vareps$-KKT point to \eqref{eq:ncp_eq} if there is a vector $\vy\in\RR^l$ such that
\begin{equation}\label{eq:kkt}
\Vert \vc(\vx) \Vert \leq \vareps,\quad
\dist\left(\vzero, \partial f_0(\vx)+J_c^\top(\vx) \  \vy \right) \leq \vareps.
\end{equation}
\end{definition}

\begin{definition}[$L$-smoothness]
A differentiable function $f$ on $\RR^n$ is $L$-smooth if $\|\nabla f(\vx_1)-\nabla f(\vx_2)\|\le L\|\vx_1-\vx_2\|$ for all $\vx_1, \vx_2\in\RR^n$. 
\end{definition}

\begin{definition}[$\rho$-weakly convex]
A function $g$ is $\rho$-weakly convex if $g+\frac{\rho}{2}\|\cdot\|^2$ is convex.
\end{definition}

\begin{remark}
If $f$ is $L$-smooth, then it is also $L$-weakly convex. However, the weak-convexity constant of a differentiable function can be much smaller than its smoothness constant. For example, if $f(\vx)=\frac{1}{2}\vx^\top\vQ\vx+\vc^\top\vx$ where $\vQ$ is a symmetric but indefinite matrix, then the smoothness constant of $f$ is $\|\vQ\|$, and its weak-convexity constant is the negative of the smallest eigenvalue of $\vQ$.
\end{remark}

Throughout the paper, we make the following assumptions about \eqref{eq:ncp_eq}. Examples that satisfy these assumptions will be given in the experimental section.
\begin{assumption}[smoothness and weak convexity]\label{assump:smooth-wc}
The function $g$ in the objective of \eqref{eq:ncp_eq} is $L_0$-smooth and $\rho_0$-weakly convex. For each $j\in[l]$, $c_j$ is $L_j$-smooth and $\rho_j$-weakly convex.
\end{assumption}

\begin{assumption}[bounded domain]\label{assump:composite}
$h$ is a simple closed convex function with a compact domain, i.e.,
\begin{equation}\label{eq:def-D}
D =: \max_{\vx,\vx'\in \dom(h)}\|\vx-\vx'\| <\infty.
\end{equation} 
\end{assumption}

\section{A NOVEL AL-BASED FOM WITH IMPROVED CONVERGENCE RATE} \label{sec:alg} 
In this section, we present a novel FOM (see Algorithm~\ref{alg:ialm} below) for solving \eqref{eq:ncp_eq}. It follows the standard ALM framework, similar to  AL-based FOMs \citep{sahin2019inexact, xu2019iter-ialm}. Notably, different from existing works, we use an inexact proximal point method (iPPM) to approximately solve each ALM subproblem. The complexity result of iPPM has the best dependence on the smoothness constant. This enables us to obtain order-reduced complexity results by geometrically increasing the penalty parameter in ALM, as compared to the AL-based FOMs \citep{li2020augmented, sahin2019inexact} for nonconvex constrained optimization. 
Our whole algorithm has three layers. We analyze the inner algorithm~\ref{alg2} in Section~\ref{sec:apg}, the middle algorithm~\ref{alg:ippm} in Section~\ref{sec:ippm}, and the outer algorithm~\ref{alg:ialm} in Section~\ref{sec:ialm}.

\begin{algorithm}[h]
	\caption{Accelerated proximal gradient method: $\mathrm{APG}(G,H,\mu, L_G,\vareps)$}\label{alg2}
	\DontPrintSemicolon
	\textbf{Initialization:} choose $\bar{\vx}^{-1}\in \dom(H)$ and set $\alpha=\sqrt{\frac{\mu}{L_G}}$;
	let 
	\begin{align*}\bar{\vx}^0=\vx^{0}=\argmin_{\vx} &~\langle \nabla G(\bar{\vx}^{-1}),\vx \rangle\\[-0.1cm]
	&~\textstyle+\frac{L_G}{2}\left\Vert \vx-\bar{\vx}^{-1}\right\Vert^2+H(\vx).\end{align*}\;
	\vspace{-0.4cm}
	\For{$t=0,1,\ldots$}{
	Update the iterate by
	    \small
		\begin{align}
		&\textstyle\vx^{t+1}=\argmin_{\vx} \langle \nabla G(\bar{\vx}^t),\vx \rangle+\frac{L_G}{2}\left\Vert \vx-\bar{\vx}^t\right\Vert^2+H(\vx), \label{alg2-1} \\
		&\textstyle\bar{\vx}^{t+1}=\vx^{t+1}+\frac{1-\alpha}{1+\alpha}(\vx^{t+1}-\vx^t). \label{alg2-2}
		\end{align}
		\normalsize
	\textbf{if} $\dist\big(-\nabla G(\vx^{t+1}), \partial H(\vx^{t+1}) \big)\leq \vareps$, \textbf{then} output $\vx^{t+1}$ and stop.	
	}
\end{algorithm}

\subsection{Accelerated proximal gradient (APG) method for convex composite problems} \label{sec:apg}

The kernel problems that we solve are a sequence of convex composite problems in the form of
\begin{equation}\label{eq:comp-prob}
\Min_{\vx\in\RR^n}~ F(\vx):=G(\vx)+H(\vx),
\end{equation}
where $G$ is $\mu$-strongly convex and $L_G$-smooth, and $H$ is a closed convex function. Various optimal FOMs (e.g., \cite{nesterov2013gradient, nesterov2004introductory, FISTA2009}) have been designed to solve \eqref{eq:comp-prob}. We choose the FOM used by \cite{li2020augmented} for the purpose of obtaining near-stationary points. Its pseudocode is given in Algorithm~\ref{alg2}.

The next lemma is from \cite[Lemma~3]{li2020augmented}. 
It gives the complexity result of Algorithm \ref{alg2}.

\begin{lemma}\label{lem:total-iter-apg} 
Given $\vareps>0$, within at most $T$ iterations, Algorithm~\ref{alg2} will output a solution $\vx^{T}$ that satisfies $\dist\big(\vzero,\partial F(\vx^{T})\big) \le \vareps$, 
where 
$$\textstyle T=\left\lceil \sqrt{\frac{L_G}{\mu}}\log\frac{64L_G^2\left(L_G\Vert \vx^{-1}-\vx^* \Vert^2+\mu\Vert \vx^*-\vx^0 \Vert^2\right)}{ \vareps^2 \mu} + 1 \right\rceil.$$
\end{lemma}

\subsection{Inexact proximal point method (iPPM) for nonconvex composite problems} \label{sec:ippm}
Each primal subproblem of the ALM for \eqref{eq:ncp_eq} is a nonconvex composite problem in the form of
\begin{equation}\label{eq:nc_prob}
\Phi^*= \Min_{\vx\in\RR^n} ~\big\{\Phi(\vx) := \phi(\vx)+\psi(\vx)\big\},
\end{equation}
where $\phi$ is $L_\phi$-smooth and $\rho$-weakly convex, and $\psi$ is closed convex. We propose to use the iPPM to approximately solve the ALM subproblems. The iPPM framework has appeared in \citep{kong2019complexity-pen}. Different from \citep{kong2019complexity-pen}, we propose to use APG in Algorithm~\ref{alg2} to solve each iPPM subproblem. The pseudocode of our iPPM is shown in Algorithm~\ref{alg:ippm}. 
It appears that our iPPM has more stable numerical performance. 

\begin{algorithm}[h] 
\caption{Inexact proximal point method (iPPM) for \eqref{eq:nc_prob}: iPPM($\phi,\psi,\vx^0,\rho, L_\phi,\vareps$)} \label{alg:ippm}
\DontPrintSemicolon
\textbf{Input:} $\vx^0\in\dom(\psi)$, smoothness $L_\phi$, weak convexity $\rho$, stationarity tolerance $\vareps$\;
\For{$k=0,1,\ldots,$}{
Let $G(\cdot)=\phi(\cdot)+\rho\Vert \cdot - \vx^k \Vert^2$\; 
Call Algorithm~\ref{alg2} to obtain $\vx^{k+1} \leftarrow \mathrm{APG}(G, \psi, \rho, L_\phi+2\rho, \frac{\vareps}{4})$\;
\textbf{if} $2\rho \Vert \vx^{k+1}-\vx^k \Vert \le \frac{\vareps}{2}$, \textbf{then} return $\vx^{k+1}$.
}
\end{algorithm}

The next theorem gives the complexity result. Its proof is given in the supplementary materials.
\begin{theorem}\label{thm:ippm-compl}
Suppose $\Phi^*$ is finite. Algorithm~\ref{alg:ippm} must stop within $T$ iterations, where
\begin{equation} \label{eq9}
\textstyle T = \left\lceil \frac{32 \rho}{\vareps^2} (\Phi(\vx^0)-\Phi^*) \right\rceil.
\end{equation}
The output $\vx^S$ must be an $\vareps$-stationary point of \eqref{eq:nc_prob}, i.e., $\dist(\vzero, \partial \Phi(\vx^S)) \le \vareps$. In addition, if $\dom(\psi)$ is compact and has diameter $D_\psi < \infty$, then the total complexity is $O\left( \frac{\sqrt{\rho L_\phi}}{\vareps^2} [\Phi(\vx^0)-\Phi^*] \log \frac{D_\psi}{\vareps} \right)$.
\end{theorem}

\begin{remark}
A similar result has been shown by \cite{kong2019complexity-pen}. It has better dependence on $L_\phi$ than that by \cite{ghadimi2016accelerated}. In addition, in the worst case, $\Phi(\vx^0)-\Phi^*$ is in the same order of $L_\phi$. However, we will see that for our case, $\Phi(\vx^0)-\Phi^*$ can be uniformly bounded when Algorithm~\ref{alg:ippm} is applied to solve subproblems of ALM even if the penalty parameter (that is proportional to the smooth constant) in the AL function geometrically increases. As a result, when the smoothness and weak convexity parameters are both $O(\vareps^{-1})$, we can obtain a total complexity of $\tilde{O}\left(\frac{\sqrt{\rho L_{\Phi}}}{\vareps^2}\right) = \tilde{O}(\vareps^{-3})$, which is better than  $\tilde{O}(\frac{L_{\Phi}^2}{\vareps^2}) = \tilde{O}(\vareps^{-4})$ obtained if the method in \citep{ghadimi2016accelerated} is applied. This is the key for us to have order-reduced complexity results, as compared to \citep{sahin2019inexact}.
\end{remark}

\subsection{Inexact augmented Lagrangian method (iALM) for nonlinear constrained problems} \label{sec:ialm}

Now we are ready to present an improved AL-based FOM for solving \eqref{eq:ncp_eq}. Different from existing AL-based FOMs, our method uses iPPM, given in Algorithm~\ref{alg:ippm}, to approximately solve each subproblem, and also its dual step size is adaptive to the primal residual. The pseudocode is shown in Algorithm~\ref{alg:ialm}. 

In the algorithm and the later analysis, we denote 
\begin{subequations}\label{P3-eqn-33}
\begin{align}
&B_0=\max_{\vx\in\dom(h)}\max \big\{|f_0(\vx)|,\left\Vert \nabla g(\vx) \right\Vert\big\}, \nonumber\\
&B_c = \max_{\vx\in\dom(h)}\Vert J_c(\vx) \Vert, \\
&B_i=\max_{\vx\in\dom(h)}\max \big\{|c_i(\vx)|,\left\Vert \nabla c_i(\vx) \right\Vert\big\}, \forall\, i \in [l], \\
& \textstyle \bar{B}_c = \sqrt{\sum_{i=1}^l B_i^2},\quad \bar{L} = \sqrt{\sum_{i=1}^lL_i^2}, \nonumber \\ 
& \rho_c = \sum_{i=1}^{l} B_i \rho_i, \quad L_c = \sum_{i=1}^{l}B_i L_i + B_i^2, \label{eq:rho_c}
\end{align}
\end{subequations}
where $\{\rho_i\}$ and $\{L_i\}$ are given in Assumption~\ref{assump:smooth-wc}. Note that the above constants are all finite under Assumptions~\ref{assump:smooth-wc} and \ref{assump:composite}, and we do not need to evaluate them exactly but only need upper bounds.


\begin{algorithm}[h] 
\caption{Inexact augmented Lagrangian method (iALM) for \eqref{eq:ncp_eq}}\label{alg:ialm}
\DontPrintSemicolon
\textbf{Initialization:} choose $\vx^0\in\dom(f_0), \vy^0 = \vzero$, $\beta_0>0$ and $\sigma>1$\;
\For{$k=0,1,\ldots,$}{
Let $\beta_k=\beta_0\sigma^k$, $\phi(\cdot)=\cL_{\beta_k}(\cdot,y^k)-h(\cdot)$, and
\begin{equation}\label{eq:hat_rho_k}
\begin{aligned}
&\hat{\rho}_k = \rho_0 + \bar L\|\vy^k\| + \beta_k \rho_c, \\ 
&\hat{L}_k = L_0 + \bar L\|\vy^k\| + \beta_k L_c.
\end{aligned}\vspace{-0.3cm}
\end{equation}\;
Call Algorithm~\ref{alg:ippm} to obtain $\vx^{k+1} \leftarrow \mathrm{iPPM}(\phi,h,\vx^k,\hat{\rho}_k, \hat{L}_k, \vareps)$\;
Update $\vy$ by
\begin{align}
\vy^{k+1} = ~\vy^k  + w_k \vc(\vx^{k+1}),\label{eq:alm-y}
\end{align}
where 
\begin{equation}\label{eq:ialm_step}
\textstyle w_k = w_0 \min\left\{1, \frac{\gamma_k }{\Vert \vc(\vx^{k+1}) \Vert}\right\}.
\end{equation}
}
\end{algorithm}


Algorithm~\ref{alg:ialm} follows the standard framework of the ALM. The existing method that is the closest to ours is the iALM by \cite{sahin2019inexact}. The main difference is that we use the iPPM to solve ALM subproblems, while \citep{sahin2019inexact} applies the FOM by \cite{ghadimi2016accelerated}. This change of subroutine, together with our new analysis, leads to order-reduced complexity results under the same assumptions. In addition, we observed from the experiments that our iPPM is more stable and more efficient on solving nonconvex subproblems than the subsolver by~\cite{sahin2019inexact}. 


\section{COMPLEXITY RESULTS}\label{sec:complexity}
In this section, we analyze the complexity result of Algorithm~\ref{alg:ialm}. In general, it is difficult to show convergence rates of AL-based FOMs on nonconvex constrained problems mainly due to two reasons. First, a stationary point of the AL function may not be (near) feasible, even a large penalty parameter is used. This is essentially different from penalty-based FOMs. 
Second, the Lagrangian multiplier cannot be bounded if the dual step size is not carefully set. We show that, with a regularity condition and a well-controlled dual step size, our AL-based FOM can circumvent both issues and achieve  best-known convergence rates.

For simplicity, we let 
\begin{equation}\label{eq:gamma-k}
\gamma_k = \frac{(\log 2)^2 \Vert \vc(\vx^0) \Vert}{(k+1)[\log(k+2)]^2},    
\end{equation}
which has been adopted by \cite{sahin2019inexact}. This choice of $\gamma_k$ will lead to a uniform bound on $\{\vy^k\}$ and simplify our analysis. More complicated analysis with general $\{\gamma_k\}$ is given in the supplementary materials.

It is impossible to find a (near) feasible solution of a general nonlinear system in polynomial time. Hence, a certain regularity condition is necessary in order to guarantee near-feasibility. Following  \citep{sahin2019inexact, lin2019inexact-pp}, we assume the regularity condition below on \eqref{eq:ncp_eq}. 
\begin{assumption}[regularity] \label{assumption:regularity}
There is some $v > 0$ such that for any $k\ge1$,
\begin{equation}\label{eq:regularity}
\textstyle v\Vert \vc(\vx^k) \Vert \le \dist \left(-J_c(\vx^k)^{\top} \vc(\vx^k), \frac{\partial h(\vx^k)}{\beta_{k-1}} \right).
\end{equation}
\end{assumption}

\begin{remark}
The intuition of using Assumption~\ref{assumption:regularity} is to ensure near feasibility of a near-stationary point to the AL function. Without any regularity conditions on the nonconvex constraints, one cannot even achieve feasibility. It is unclear whether Assumption~\ref{assumption:regularity} is stronger or weaker than other common regularity conditions such as the Slater's condition and the MFCQ condition.

Several nonconvex examples that satisfy the regularity condition are given in \citep{sahin2019inexact} and \citep{lin2019inexact-pp}, such as the EV and clustering problems in our experiments. 
In Section~\ref{subsec:reg_ex}, we prove that Assumption~\ref{assumption:regularity} holds for all affine equality constrained problems possibly with an additional polyhedral constraint set or a ball constraint set. Hence, the LCQP problem in our experiments also has this property. 
Notice that we only require the existence of $v$ but do not need to know its value in our algorithm. 

\end{remark}

\subsection{Convex constraint examples with regularity condition} \label{subsec:reg_ex}
In this subsection, we show that the regularity condition in Assumption~\ref{assumption:regularity} can hold for the LCQP problem \eqref{eq:ncQP} that we will test. We prove this for a broader class of problems, namely, affine-equality constraints problems with an additional polyhedral constraint set or $\{\vx\in\RR^n: \vA\vx=\vb, \|\vx\|\le 1\}$. The proofs are given in the supplementary materials.

\subsubsection{polyhedral constraint}
Let $X\subseteq\RR^n$ be a compact polyhedral set and $h(\cdot)=\iota_X(\cdot)$ be the indicator function on $X$. Then for any $\beta>0$ and $\vx\in X$, $\frac{\partial h(\vx)}{\beta}=\cN_X(\vx)$, where $\cN_X$ denotes the normal cone. We have the result in the claim below.
\begin{claim}\label{clm:poly}
	If $X\cap \{\vx\in\RR^n: \vA\vx=\vb\}\neq\emptyset$, then there is a constant $v>0$ such that $\forall\, \vx\in X,$
	\begin{equation}\label{eq:reg-cond-poly}
	v \|\vA\vx-\vb\| \le \dist\left(\vzero, \vA^\top(\vA\vx-\vb)+\cN_X(\vx)\right),
	\end{equation}
	which implies \eqref{eq:regularity} with $\vc(\vx) = \vA\vx-\vb$ and $h(\vx)=\iota_X(\vx)$. 
\end{claim}
By this claim, we let $X=\{\vx\in\RR^n: l_i\le x_i \le u_i, \forall\, i\in [n]\}$ and immediately have that the LCQP problem \eqref{eq:ncQP} satisfies the regularity condition in Assumption~\ref{assumption:regularity}.

\subsubsection{ball constraint}
Let $X=\{\vx\in\RR^n: \|\vx\| \le r\}$ be a ball of radius $r>0$ and $h$ be the indicator function on $X$. Then we have the following result.
\begin{claim}\label{clm:ball}
	Suppose $\vA$ has full row-rank. In addition, there exists a $\hat\vx$ in the interior of $X$ such that $\vA\hat\vx = \vb$. Then there is a constant $v>0$ such that \eqref{eq:reg-cond-poly} holds.
\end{claim}

\subsection{Main Theorems}
We give the main convergence results in this subsection. Detailed proofs are provided in the supplementary materials.
\begin{theorem}[total complexity of iALM]\label{thm:ialm-compl}
Suppose that all conditions in Assumptions~\ref{assump:smooth-wc} through \ref{assumption:regularity} hold. Given $\vareps>0$, then Algorithm~\ref{alg:ialm} with $\gamma_k$ given in \eqref{eq:gamma-k} needs $\tilde{O}(\vareps^{-3})$ $\mathrm{APG}$ iterations to produce an $\vareps$-KKT solution of \eqref{eq:ncp_eq}. In addition, if $\vc(\vx)=\vA\vx-\vb$, then $\tilde{O}(\vareps^{-\frac{5}{2}})$ $\mathrm{APG}$ iterations are needed to produce an $\vareps$-KKT solution of \eqref{eq:ncp_eq}. 
\end{theorem}

\begin{remark}
For nonconvex-constrained cases, our complexity result $\tilde{O}(\vareps^{-3})$ is better than the result  $\tilde{O}(\vareps^{-4})$ obtained in \citep{sahin2019inexact} for an first-order iALM, and it matches the complexity result of a penalty-based FOM by \cite{lin2019inexact-pp}. For the affine equality-constrained case with a composite objective, our result $\tilde{O}(\vareps^{-\frac{5}{2}})$ is better than $\tilde{O}(\vareps^{-3})$ obtained by \cite{li2020augmented} for a first-order iALM  and matches the result of the penalty-based FOM by \cite{lin2019inexact-pp}. Numerically, the iALM-based FOM usually significantly outperforms a penalty-based FOM, as shown in our experiments.
\end{remark}

In Theorem~\ref{thm:ialm-compl}, we required the dual step size $\textstyle w_k = w_0 \min\left\{1, \frac{\gamma_k }{\Vert \vc(\vx^{k+1}) \Vert}\right\}$, where as in \eqref{eq:gamma-k}, 
\begin{equation*}
\gamma_k = \frac{(\log 2)^2 \Vert \vc(\vx^0) \Vert}{(k+1)[\log(k+2)]^2}.
\end{equation*}
Numerically, we observed better performance by slightly deviating from this setting. For example, we set $w_k = \frac{1}{\Vert \vc(\vx^{k+1}) \Vert}$ in all of our trials. This motivates us to give a more general version of Theorem~\ref{thm:ialm-compl}. The following theorem considers $w_k = \frac{O(k^q)}{\vc(\vx^{k+1})}$ and sacrifices an order of 
$(\log \vareps^{-1})^{q+1}$ in the total complexity compared to Theorem~\ref{thm:ialm-compl}. 

\begin{theorem}[complexity of iALM with general dual step sizes]\label{thm:ialm-cplx-2}
	In Algorithm~\ref{alg:ialm}, for some fixed $q \in \ZZ_+ \cup \{0\}$ and $M > 0$, let 
	\begin{equation}\label{wk-general}
	w_k = \frac{M(k+1)^q}{\Vert \vc(\vx^{k+1}) \Vert}, \forall k \ge 0.
	\end{equation}
	Assume all other conditions of Theorem \ref{thm:ialm-compl} hold.
	Then given $\vareps>0$, Algorithm~\ref{alg:ialm} with $w_k$ given in \eqref{wk-general} needs $\tilde{O}(\vareps^{-3})$ $\mathrm{APG}$ iterations to produce an $\vareps$-KKT solution of \eqref{eq:ncp_eq}. In addition, if $\vc(\vx)=\vA\vx-\vb$, then $\tilde{O}(\vareps^{-\frac{5}{2}})$ $\mathrm{APG}$ iterations are needed to produce an $\vareps$-KKT solution of \eqref{eq:ncp_eq}. 
\end{theorem}

The proof of the above theorem is very similar to the proof of Theorem \ref{thm:ialm-compl}, except we have a nonuniform bound on the dual variable.

\begin{remark}[inequality constraints]\label{rm:ineq-case}
Although only equality constraints are considered in \eqref{eq:ncp_eq}, our complexity result does not lose generality due to the boundedness of $\{\vy^k\}$. Suppose we solve a problem with both equality and inequality constraints
\begin{equation}\label{eq:orig}
\Min_\vx f_0(\vx), \st \vc(\vx)=\vzero, \vd(\vx)\le \vzero.    
\end{equation} 
Introducing a slack variable $\vs\ge \vzero$, we can have an equivalent formulation 
\begin{equation}\label{eq:reform}
\Min_{\vx, \vs\ge\vzero} f_0(\vx), \st \vc(\vx)=\vzero, ~\vd(\vx) + \vs = \vzero.    
\end{equation}
Suppose the conditions required by Theorem~\ref{thm:ialm-compl} hold. Then we can apply Algorithm~\ref{alg:ialm} to \eqref{eq:reform} and obtain an $\vareps$-KKT point $(\bar\vx,\bar\vs)$ with a corresponding multiplier $(\bar\vy,\bar\vz)$, i.e.,
\begin{subequations} \small
\begin{align}
\dist\left(\vzero, \left[\begin{array}{c}\partial f_0(\bar\vx)\\\cN_+(\bar\vs)\end{array}\right]+\left[\begin{array}{c}J_c(\bar\vx)^\top\\ \vzero\end{array}\right]\bar\vy+\left[\begin{array}{c}J_d(\bar\vx)^\top\\\vI\end{array}\right]\bar\vz\right) \le \vareps, \label{eq:kkt-ineq-eq-d}\\ 
\|\vc(\bar\vx)\|^2 + \|\vd(\bar\vx)+\bar\vs\|^2 \le \vareps^2, \ \bar\vs\ge \vzero, \label{eq:kkt-ineq-eq-p}
\end{align}
\end{subequations} \normalsize
where $\cN_+(\vs)$ denotes the normal cone of the nonnegative orthant at $\vs$.

By \eqref{eq:kkt-ineq-eq-d} and the definition of the normal cone, we have $\|[\bar\vz]_-\|\le \vareps$. Let $\hat\vz = \bar\vz-[\bar\vz]_-$. Then $\hat\vz\ge\vzero$, and if $\|J_d(\cdot)\|$ is uniformly bounded, then it follows from \eqref{eq:kkt-ineq-eq-d} that
\begin{equation}\label{eq:reform-d}
\dist\left(\vzero, \partial f_0(\bar\vx) J_c(\bar\vx)^\top \bar\vy + J_d(\bar\vx)^\top\hat\vz\right) = O(\vareps).  
\end{equation}
In addition, from \eqref{eq:kkt-ineq-eq-p}, it is straightforward to have $\|\vc(\bar\vx)\|^2+\|[\vd(\bar\vx)]_+\|^2\le \vareps^2$. Furthermore, notice that if some $\bar s_i =0$, then $|d_i(\bar\vx)|\le \vareps$ from \eqref{eq:kkt-ineq-eq-p}, and if $\bar s_i >0$, then $|\bar z_i| \le \vareps$ from \eqref{eq:kkt-ineq-eq-d}. Finally, use the boundedness of $\vd$ and the fact that $\|\bar\vz\| = O(1)$ is independent of $\vareps$ from the proof of Theorem~\ref{thm:ialm-compl} to have $|\hat\vz^\top \vd(\bar\vx)| = O(\vareps)$. Therefore, $\bar\vx$ is an $O(\vareps)$-KKT point of the original problem \eqref{eq:orig}, in terms of primal feasibility, dual feasibility, and the complementarity condition. 
\end{remark}

We have established an $\tilde{O}(\vareps^{-\frac{5}{2}})$ complexity for problems with affine-equality constraints. In Appendix \ref{sec:cvx_ineq}, we extend the same-order complexity result to problems with general convex constraints, namely, with affine-equality constraints and also convex inequality constraints.

\section{NUMERICAL RESULTS} \label{sec:experiment}

In this section, we conduct experiments to demonstrate the empirical performance of the proposed improved iALM. We consider the nonconvex linearly-constrained quadratic program (LCQP), generalized eigenvalue problem (EV), and clustering problem.
We compare our method to the iALM by \cite{sahin2019inexact} for all three problems, and also to HiAPeM by \cite{li2020augmented} for the LCQP problem. 
All the tests were performed in MATLAB 2019b on a Macbook Pro with 4 cores and 16GB memory. 
Due to the page limitation, we put some tables with more details in the supplementary materials.

\subsection{Nonconvex linearly-constrained quadratic programs (LCQP)} \label{sec:experiment-ncvx}

In this subsection, we test the proposed method on solving nonconvex LCQP:
\begin{equation}\label{eq:ncQP} 
\begin{aligned}
&\textstyle \min_{\vx \in \mathbb{R}^n}  \frac{1}{2} \vx^\top \vQ \vx + \vc^\top  \vx,\\ 
&\text{s.t. } \vA\vx=\vb,\ x_i \in [l_i,u_i],\,\forall\, i\in [n],
\end{aligned}
\end{equation} \normalsize
where $\vA \in \mathbb{R}^{m\times n}$, and $\vQ\in\RR^{n\times n}$ is symmetric and indefinite (thus the objective is nonconvex). In the test, we generated all data randomly. The smallest eigenvalue of $\vQ$ is $-\rho < 0$, and thus the problem is $\rho$-weakly convex. For all tested instances, we set $l_i=-5$ and $u_i=5$ for each $i\in [n]$.

We generated two groups of LCQP instances of different sizes. The first group had $m = 10$ and $n = 200$ and the second one $m = 100$ and $n = 1000$.
In each group, we generated $10$ instances of LCQP with $\rho = 1$. 
We compared the improved iALM in Algorithm~\ref{alg:ialm} to the iALM by \cite{sahin2019inexact} and the HiAPeM method by \cite{li2020augmented}. HiAPeM adopted a hybrid setting ($N_0 = 10, N_1 = 2$) and a pure-penalty setting ($N_0 = 1, N_1 = 10^6$), where $N_0$ is the number of initial iALM calls, and afterwards, $N_1$ is roughly the number of penalty method calls before each iALM call). The AL function $\cL_{\beta}(\cdot,\vy)$ of LCQP is $\Vert \vQ+\beta \vA^{\top}\vA \Vert$-smooth and $\rho$-weakly convex. 
We set $\beta_k=\sigma^k\beta_0$ with $\sigma=3$ and $\beta_0=0.01$ for both iALMs. For the subsolver of the iALM by \cite{sahin2019inexact}, we set its step size to $\frac{1}{2 \Vert \vQ+\beta_k \vA^{\top}\vA \Vert}$ for the $k$-th outer iteration, as specified by~\cite{ghadimi2016accelerated}. The 
tolerance was set to $\vareps = 10^{-3}$ for all instances. 
In addition, we set the maximum inner iteration to $10^6$ for all methods.

In the top row of Figure~\ref{fig:all_plot},
we compare the primal residual trajectories of our method and the iALM by \cite{sahin2019inexact} on one representative instance of \eqref{eq:ncQP}. Note the dual residuals of both methods are below error tolerance at the end of each outer loop. 
In Tables \ref{table:qp-small-avg} and \ref{table:qp-large-avg}, we report, for each method, the primal residual, dual residual, running time (in seconds), and the number of gradient evaluation, shortened as \verb|pres|, \verb|dres|, \verb|time|, and \verb|#Grad|, averaged across all ten trials. The complete tables are given in the appendix.

From the results, we conclude that, to reach an $\vareps$-KKT point to the LCQP problem, the proposed improved iALM needs significantly fewer gradient evaluations and takes far less time than all other compared methods.

\setlength{\tabcolsep}{3pt}

\begin{table}\caption{Results by the proposed improved iALM, the iALM by \cite{sahin2019inexact}, and the HiAPeM by \cite{li2020augmented} on solving a $1$-weakly convex LCQP \eqref{eq:ncQP} of size $m=10$ and $n=200$. }\label{table:qp-small-avg} 
\begin{center}
\resizebox{0.48 \textwidth}{!}{ 
\begin{tabular}{|c||cccc|cccc|cccc|cccc|} 
\hline
method & pres & dres & time & \#Grad 
\\\hline\hline 
proposed improved iALM & 4.10e-4 & 5.49e-4 & 1.44 & 34294 \\ iALM in \citep{sahin2019inexact} & 5.26e-4 & 1.00e-3 & 11.03 & 1235210 \\
HiAPeM with $N_0 = 10, N_1 = 2$ & 1.97e-4 & 7.53e-4 & 2.71 & 172395 \\ HiAPeM with $N_0 = 1, N_1 = 10^6$ & 3.52e-4 & 8.20e-4 & 6.08 & 493948 \\\hline
\end{tabular}
}
\end{center}
\end{table}

\begin{table}\caption{Results by the proposed improved iALM, the iALM by \cite{sahin2019inexact}, and the HiAPeM by \cite{li2020augmented} on solving a 1-weakly convex LCQP \eqref{eq:ncQP} of size $m=100$ and $n=1000$.}\label{table:qp-large-avg} 
\begin{center}
\resizebox{0.48 \textwidth}{!}{
\begin{tabular}{|c||cccc|cccc|cccc|cccc|} 
\hline
method & pres & dres & time & \#Grad 
\\\hline\hline 
proposed improved iALM & 5.57e-4 & 8.81e-4 & 135.47 & 278395 \\ iALM in \citep{sahin2019inexact} & 4.45e-4 & 3.37e-3 & 1782.6 & 11186171 \\
HiAPeM with $N_0 = 10, N_1 = 2$ & 3.61e-4 & 8.16e-4 & 585.84 & 3081631 \\ HiAPeM with $N_0 = 1, N_1 = 10^6$ & 5.46e-4 & 9.01e-4 & 991.54 & 5738336 \\\hline
\end{tabular}
}
\end{center}
\end{table}

\subsection{Generalized eigenvalue problem}
In this subsection, we consider the generalized eigenvalue problem (EV) and compare our method to the iALM by \cite{sahin2019inexact}. 

The EV problem is
\begin{equation}\label{eq:EV}
\textstyle \min_{\vx \in \mathbb{R}^n}  \vx^\top \vQ \vx ,\ 
\text{s.t. } \vx^\top \vB \vx - 1 = 0,
\end{equation}
where $\vQ, \vB \in \mathbb{R}^{n\times n}$ are symmetric, and $\vB$ is positive definite. In the test, we set $\vQ = \frac{1}{2}(\hat{\vQ}+\hat{\vQ}^{\top})$ with the entries of $\hat{\vQ}$ independently following from the standard Gaussian $\cN(0,1)$. To ensure $\vB$ to be positive definite, we set $\vB = \bar{\vB} + (\Vert \bar{\vB} \Vert+1) \vI_{n \times n}$, where $\bar{\vB}$ is generated in the same way as $\vQ$. 
The regularity condition in Assumption \ref{assumption:regularity} has been shown for \eqref{eq:EV} by \cite{sahin2019inexact}. However, we do not have an explicit compact constraint set, i.e., Assumption~\ref{assump:composite} is violated. Nevertheless, the feasible region of \eqref{eq:EV} is bounded because of the positive definiteness of $\vB$. Hence, the tested methods can still perform well.  

Again, we generated two groups of instances of \eqref{eq:EV}, one with $n = 200$ and the other $n = 1000$. Each group consisted of 10 instances. For \eqref{eq:EV}, we were unable to obtain an explicit formula of the smoothness constant $L_k$ and weak convexity constant $\rho_k$ of the AL function $\cL_{\beta_k}(\cdot,\vy^k)$ for any $k$. The iALM by \cite{sahin2019inexact} used the accelerated first-order method by \cite{ghadimi2016accelerated} as a subroutine. We observed divergence by performing line search to estimate a local smoothness constant. In order to make it converge, we tuned its smoothness constant to $L_k = 2\Vert \vQ \Vert + 1000 + 100 \beta_k$ when $n = 200$, and $L_k = 2\Vert \vQ \Vert + 100000 + 10000  \beta_k$ when $n = 1000$. 
The weak convexity constant was tuned to $\rho_k = -0.2 \cdot \min(\text{eig}(\vQ))+\beta_k$.

In the middle row of Figure~\ref{fig:all_plot},
we compare the primal residual trajectories of our method and the iALM by \cite{sahin2019inexact} on one instance of \eqref{eq:EV}. From the results, we conclude that, to reach an $\vareps$-KKT point to the EV problem, the proposed improved iALM takes significantly fewer gradient evaluations (thus shorter time) than the iALM by \cite{sahin2019inexact}, for both small-sized and large-sized instances.

\begin{figure}[t] 
	\begin{center}
		\resizebox{0.5 \textwidth}{!}{
		\begin{tabular}{cc}
			instance of LCQP \eqref{eq:ncQP} & instance of LCQP \eqref{eq:ncQP}\\
			size $m = 20$ and $n = 100$ & size $m = 100$ and $n = 1000$\\[-0.1cm]
	\includegraphics[width=0.35\textwidth]{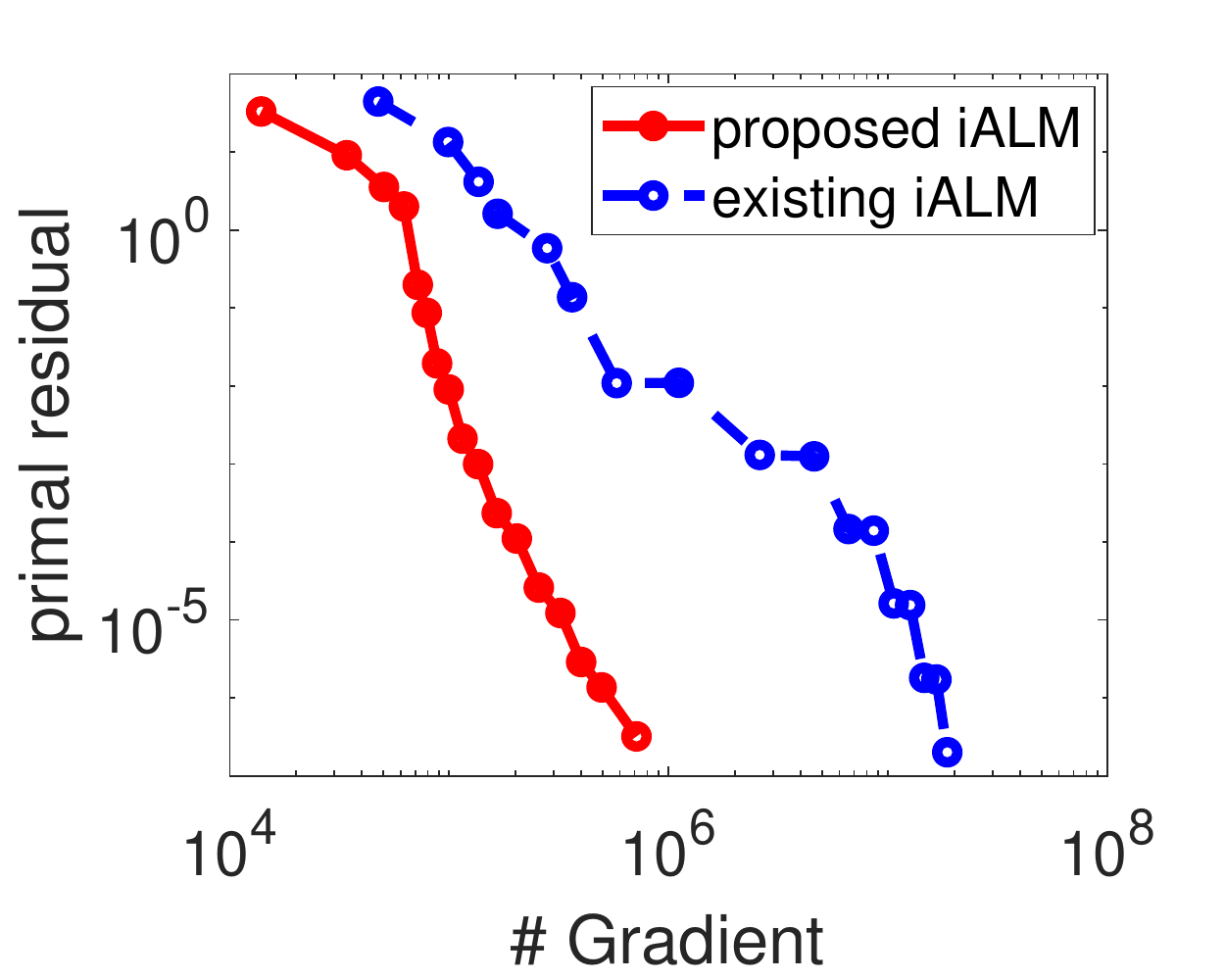} & 
			\includegraphics[width=0.35\textwidth]{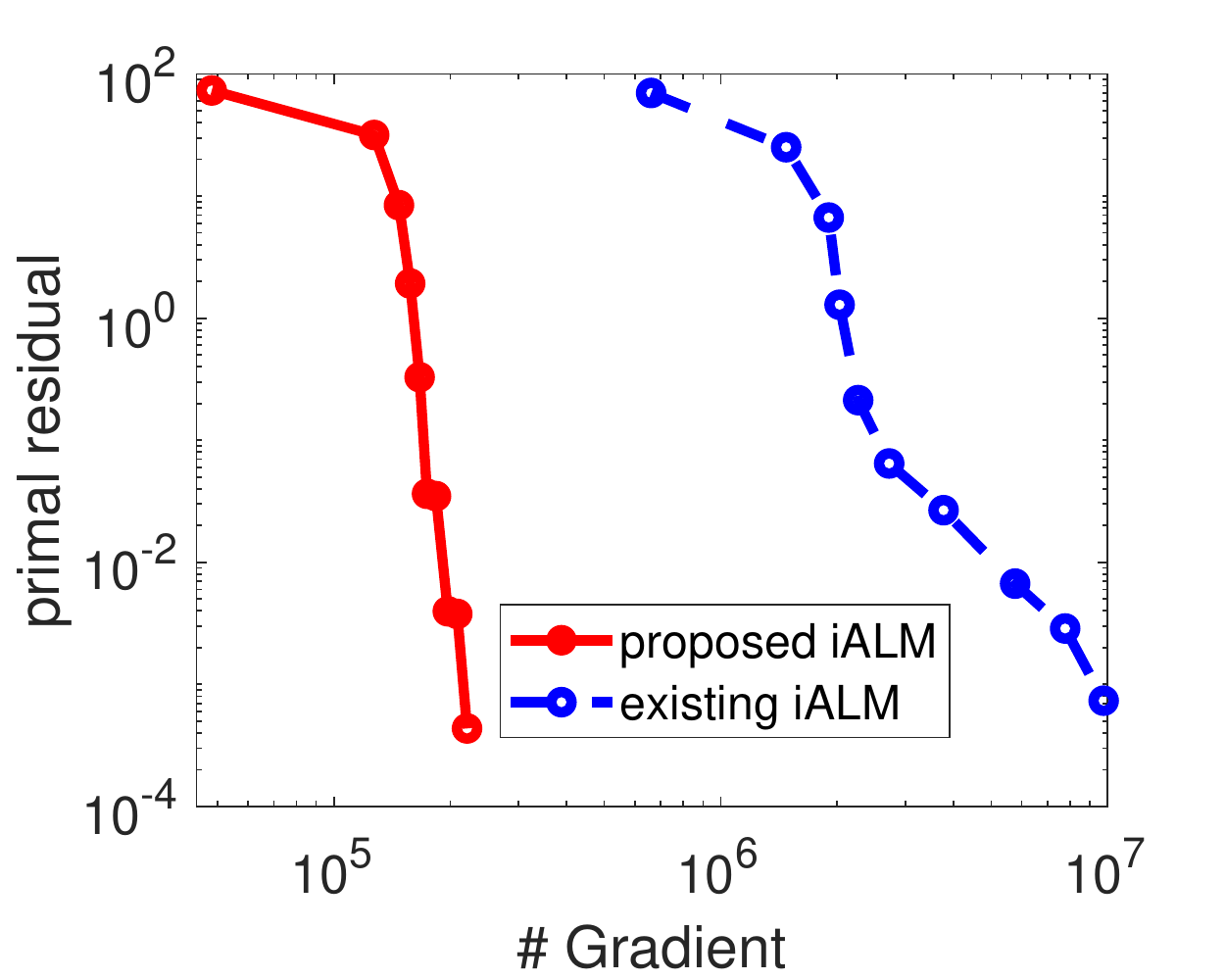} \\[0.3cm]
			instance of EV \eqref{eq:EV} & instance of EV \eqref{eq:EV} \\
			size $n = 200$ & size $n = 1000$\\[-0.1cm]
			\includegraphics[width=0.35\textwidth]{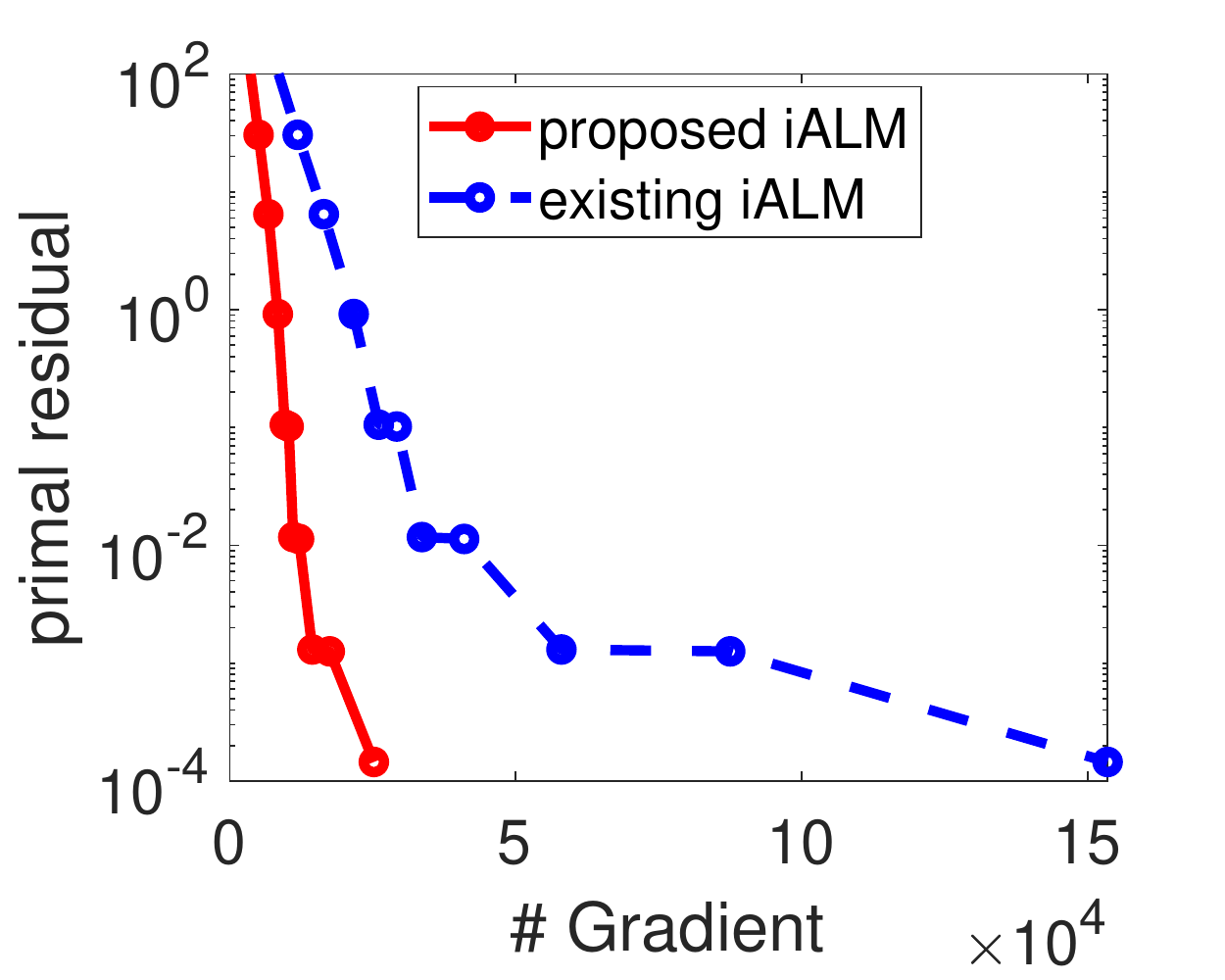} & 
			\includegraphics[width=0.35\textwidth]{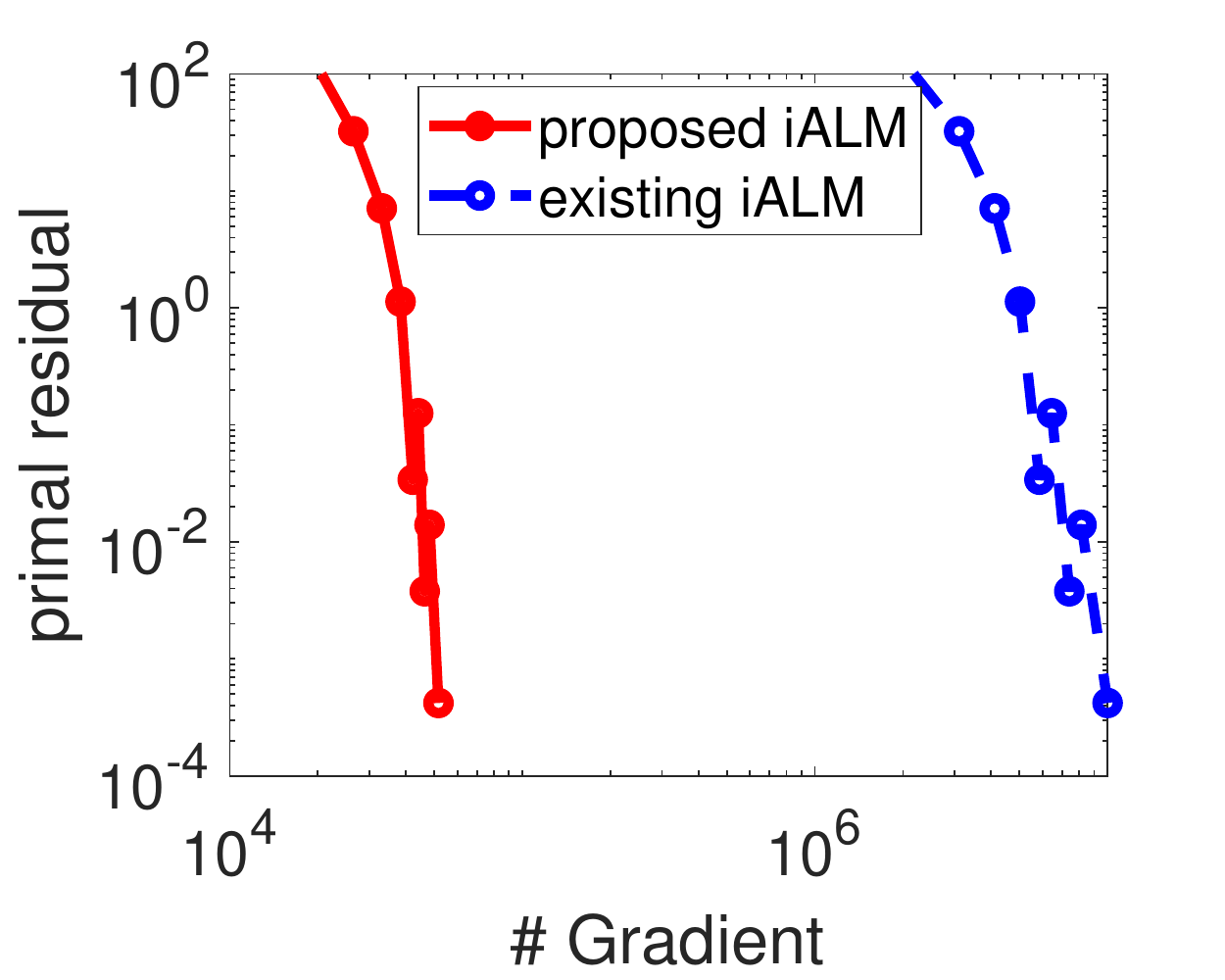}\\[0.3cm]
			clustering \eqref{eq:cluster} with Iris data & clustering \eqref{eq:cluster} with Spambase data\\
            size $(n,s,r) = (150,100,6)$ & size $(n,s,r) = (1000,100,4)$\\[-0.1cm]
			\includegraphics[width=0.35\textwidth]{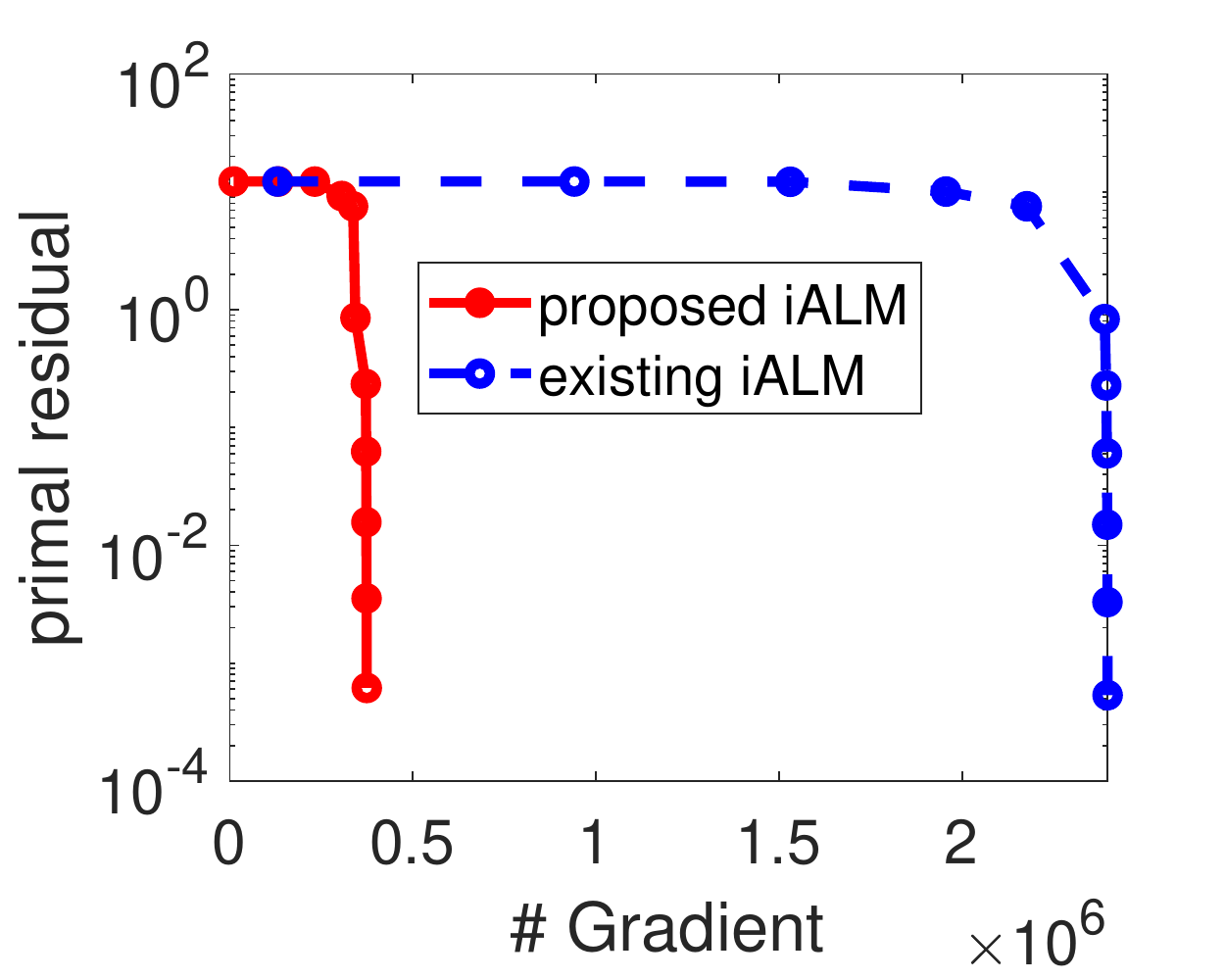} & 
			\includegraphics[width=0.35\textwidth]{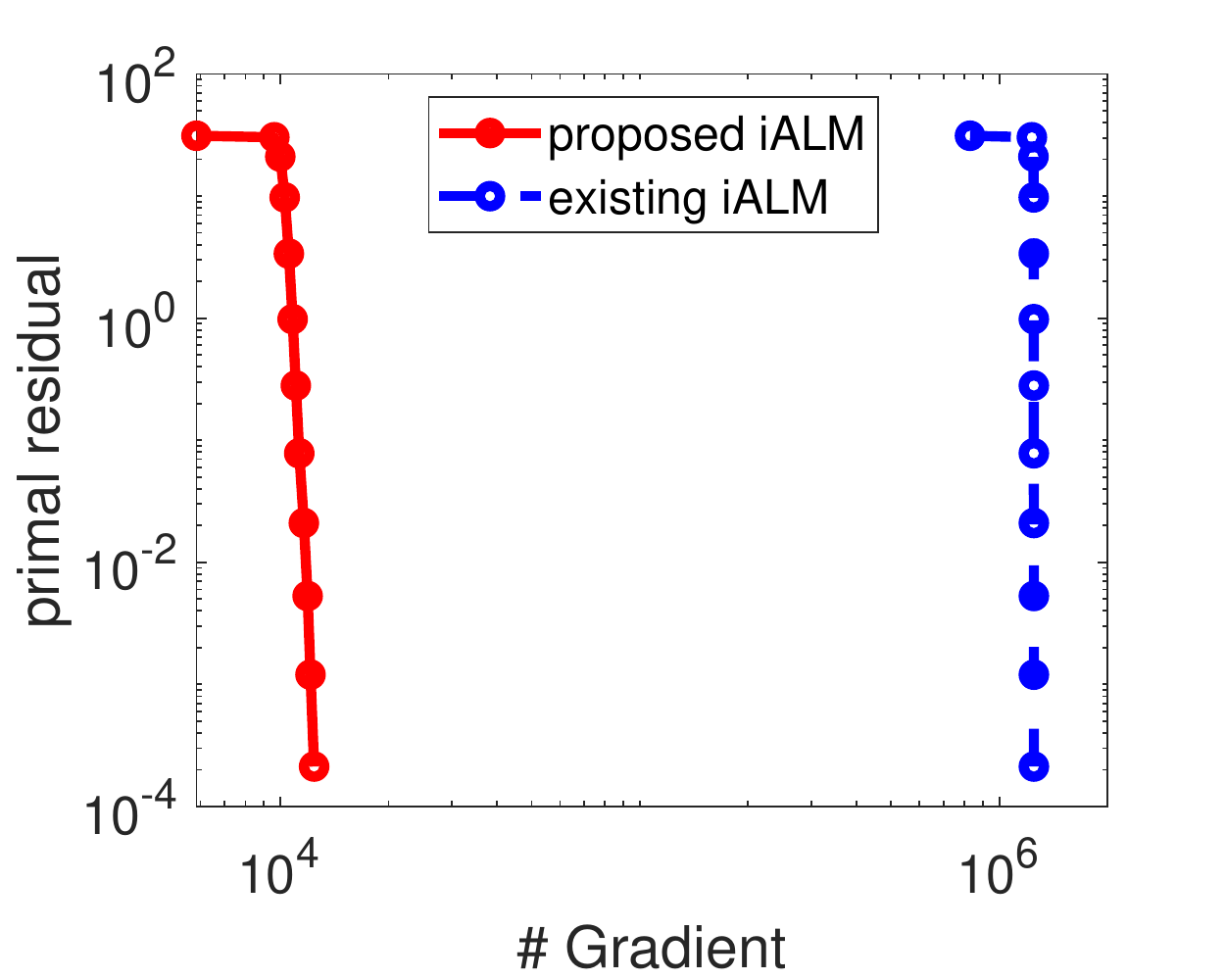} \\
		\end{tabular}
	}
	\end{center}
	\vspace{-0.2cm}
		\caption{Comparison of the proposed iALM and the existing iALM in \citep{sahin2019inexact} on solving the LCQP (first row), EV (middle row), and clustering (last row) problems. Each column shows different problem configurations and dimensions as indicated in the plot title. Each plot shows the primal residual. The markers denote the outer iterations in iALM.
	Dual residuals for both methods are similar, below a given tolerance $\vareps$.}\label{fig:all_plot}
\end{figure}

\subsection{Clustering problem}
In this subsection, we consider the clustering problem proposed in~\citep{sahin2019inexact} and compare our method to the iALM by \cite{sahin2019inexact}. 
The clustering problem is formulated as
\begin{equation}\label{eq:cluster}
\begin{aligned}
&\textstyle \min_{\vX \in C}   \sum_{i,j=1}^n D_{i,j} \langle \vx_i,\vx_j \rangle ,\\ 
&\text{s.t. }  \textstyle\vx_i^\top \sum_{j=1}^{n}\vx_j-1 = 0,\forall\, i=1,\ldots,n ,
\end{aligned}
\end{equation}
where $\vX = [\vx_1^\top, \cdots, \vx_n^\top]^\top$, $C$ is the intersection of the positive orthant with the Euclidean ball of radius $s$, and $\vD$ is the distance matrix generated by a set of data points $\{\vz_i\}_{i=1}^n$ such that $D_{i,j} = \Vert \vz_i-\vz_j \Vert$.

We run two instances, on Iris data set \citep{Dua:2019} with $n = 150, s = 100, r = 6$ and Spambase data set \citep{Dua:2019} with $n = 1000, s = 100, r = 4$. 
We tuned the smoothness constant of iALM by \cite{sahin2019inexact} to $L_k = 80 \Vert \vD \Vert + 1200 \beta_k$ in order to have convergence. 
The weak convexity constant was tuned to $\rho_k = -0.2r \cdot \min(\text{eig}(\vD)) \cdot \beta_k$.

In the bottom row of Figure~\ref{fig:all_plot},
we compare the primal residual trajectories of our method and the iALM by \cite{sahin2019inexact}. 
From the results, we conclude that, to reach an $\vareps$-KKT point to the clustering problem, the proposed improved iALM needs significantly fewer gradient evaluations (thus shorter time) than the iALM by \cite{sahin2019inexact}. The advantage of our method is even more significant for the larger-sized instance.

\setlength{\tabcolsep}{3pt}

\section{CONCLUSION}
We have presented an improved iALM for solving nonconvex constrained optimization. Different from existing iALMs, our iALM uses the iPPM to approximately solve each subproblem. Under the same regularity condition as existing works, we explore the better convergence rate of iPPM and the boundedness of AL functions to establish improved complexity results. To reach an $\vareps$-KKT solution, our method requires $\tilde{O}(\vareps^{-\frac{5}{2}})$ proximal gradient steps for solving nonconvex optimization with affine-equality constraints. The result is slightly worsened to $\tilde{O}(\vareps^{-3})$ if the constraints are also nonconvex. 
Both complexity results are so far the best. Numerically, we demonstrated that the proposed improved iALM could significantly outperform one existing iALM and also one penalty-based FOM.

\section*{Acknowledgements}

$^*$The authors Pin-Yu Chen, Sijia Liu,  Songtao Lu, and Yangyang Xu are listed in alphabetical order.

This work was supported by the Rensselaer-IBM AI Research Collaboration (\url{http://airc.rpi.edu}), part of the IBM AI Horizons Network (\url{http://ibm.biz/AIHorizons}).


\bibliography{optim}
\bibliographystyle{plainnat}

\newpage

\appendix



\onecolumn

\section{PROOFS}
In this section, we provide detailed proofs of our theorems.
\subsection{Proof of Theorem~\ref{thm:ippm-compl}}
	Let $\Phi_k(\vx) := \Phi(\vx)+\rho \Vert \vx-\vx^k \Vert^2$ and $\Phi_k^*=\min_\vx\Phi_k(\vx)$ for each $k\ge0$. 
	Note we have $\dist(\vzero, \partial \Phi_k(\vx^{k+1})) \le \delta=\frac{\vareps}{4}$, and also $\Phi_k$ is $\rho$-strongly convex. Hence $\Phi_k(\vx^{k+1}) - \Phi_k^* \le \frac{\delta^2}{2\rho}$, and $\Phi(\vx^{k+1})+\rho\Vert\vx^{k+1}-\vx^k\Vert^2-\Phi(\vx^k) \le \frac{\delta^2}{2\rho}$. Thus,
	\begin{align}
	&\Phi(\vx^T)-\Phi(\vx^0)+\rho\sum_{k=0}^{T-1}\Vert \vx^{k+1}-\vx^k \Vert^2 \le \frac{T\delta^2}{2\rho} \nonumber\\
	&T\min_{0 \le k \le T-1} \Vert \vx^{k+1}-\vx^k \Vert^2 \le \frac{1}{\rho} \left(\frac{T \delta^2}{2\rho}+[\Phi(\vx^0)-\Phi(\vx^T)]\right) \nonumber\\
	&2\rho \min_{0 \le k \le T-1} \Vert \vx^{k+1}-\vx^k \Vert \le 2\sqrt{\frac{\delta^2}{2}+\frac{\rho [\Phi(\vx^0)-\Phi^*] }{T}}. \label{eq:0}
	\end{align}
	Since $T \ge \frac{32\rho}{\vareps^2}[\Phi(\vx^0)-\Phi^*]$ and $\delta = \frac{\vareps}{4}$, we have
	\begin{equation} \label{eq:2}
	\frac{\rho}{T}[\Phi(\vx^0)-\Phi^*] \le \frac{\vareps^2}{32},
	\end{equation}
	and thus \eqref{eq:0} implies
	\begin{equation}\label{2}
	2\rho \min_{0 \le k \le T-1} \Vert \vx^{k+1}-\vx^k \Vert \le \frac{\vareps}{2}.
	\end{equation}
	Therefore, Algorithm~\ref{alg:ippm} must stop within $T$ iterations, from its stopping condition, and when it stops, the output $\vx^S$ satisfies $2\rho \Vert \vx^{S}-\vx^{S-1} \Vert \le \frac{\vareps}{2}$.
	
	Now recall $\dist(0, \partial \Phi_k(\vx^{k+1})) \le \delta = \frac{\vareps}{4}$, i.e.,
	\begin{equation}\label{1}
	\dist(0, \partial \Phi(\vx^{k+1}) + 2\rho(\vx^{k+1}-\vx^k)) \le \frac{\vareps}{2}, \forall k \ge 0.
	\end{equation}
	The above inequality together with $2\rho \Vert \vx^{S}-\vx^{S-1} \Vert \le \frac{\vareps}{2}$ gives
	\begin{equation*}
	\dist(\vzero, \partial \Phi(\vx^{S})) \le \vareps,
	\end{equation*}
	which implies that $\vx^{S}$ is an $\vareps$-stationary point to \eqref{eq:nc_prob}.
	
	Finally, we apply Lemma~\ref{lem:total-iter-apg} to obtain the overal complexity and complete the proof.

\subsection{Proof of Claim~\ref{clm:poly}}	
	Let $X_*$ be the optimal solution set of
	\begin{equation}
	\min_{\vx\in X} f(\vx):=\frac{1}{2}\|\vA\vx-\vb\|^2.
	\end{equation}
	Then for any $\bar\vx\in X_*$, $\vA\bar\vx-\vb=\vzero$ by our assumption.
	From \cite[Theorem 18]{wang2014iteration}, it follows that there is a constant $\kappa>0$ such that
	\begin{equation}\label{eq:error-bd}
	\|\vx -\Proj_{X_*}(\vx)\| \le \kappa \left\|\vx -\Proj_X\big(\vx - \nabla f(\vx)\big)\right\|, \, \forall\, \vx\in X,   
	\end{equation}
	where $\Proj_X$ denotes the Euclidean projection onto $X$.
	
	For any fixed $\vx\in X$, denote $\vu = \nabla f(\vx)$ and $\vv=\Proj_X(\vx - \vu)$. Then from the definition of the Euclidean projection, it follows that $\langle \vv - \vx+\vu, \vv - \vx'\rangle \le 0, \,\forall\, \vx'\in X$. Letting $\vx'= \vx$, we have $\|\vv-\vx\|^2 \le \langle \vu, \vx-\vv\rangle$. On the other hand, for any $\vz \in \cN_X(\vx)$, we have from the definition of the normal cone that $\langle \vz , \vx - \vx'\rangle \ge 0,\,\forall\, \vx'\in X$. Hence, letting $\vx'=\vv$ gives $\langle \vz , \vx - \vv\rangle \ge 0$. Therefore, we have
	$$\|\vv-\vx\|^2 \le \langle \vu, \vx-\vv\rangle + \langle \vz , \vx - \vv\rangle \le \|\vx - \vv\|\cdot\|\vu+\vz\|,$$
	which implies $\|\vv-\vx\|\le \|\vu+\vz\|$. By the definition of $\vu$ and $\vv$ and noticing that $\vz$ is an arbitrary vector in $\cN_X(\vx)$, we obtain
	$$\left\|\vx -\Proj_X\big(\vx - \nabla f(\vx)\big)\right\| \le \dist\left(\vzero, \nabla f(\vx) + \cN_X(\vx)\right).$$
	The above inequality together with \eqref{eq:error-bd} gives
	\begin{equation}\label{eq:bd-to-dist}\|\vx -\Proj_{X_*}(\vx)\| \le \kappa\cdot \dist\left(\vzero, \nabla f(\vx) + \cN_X(\vx)\right), \, \forall\, \vx\in X.
	\end{equation}
	
	Now by the fact $\vA \Proj_{X_*}(\vx) = \vb$, we have $\|\vA\vx-\vb\|\le \|\vA\|\cdot \|\vx -\Proj_{X_*}(\vx)\|$. Therefore, from \eqref{eq:bd-to-dist} and also noting $\nabla f(\vx)=\vA^\top(\vA\vx-\vb)$, we obtain \eqref{eq:reg-cond-poly} with $v= \frac{1}{\kappa\|\vA\|}$.

\subsection{Proof of Claim~\ref{clm:ball}}
	Without loss of generality, we assume $r=1$ and $\vA\vA^\top = \vI$, i.e., the row vectors of $\vA$ are orthonormal. Notice that
	\begin{equation}\label{eq:n-cone-ball}
	\cN_X(\vx)=\left\{
	\begin{array}{ll}
	\{\vzero\}, &\text{ if }\|\vx\| < 1,\\[0.1cm]
	\{\lambda \vx: \lambda\ge0 \}, & \text{ if }\|\vx\| = 1.
	\end{array}
	\right.
	\end{equation}
	Hence, if $\|\vx\| < 1$, \eqref{eq:reg-cond-poly} holds with $v=1$ because $\vA\vA^\top = \vI$. In the following, we focus on the case of $\|\vx\| = 1$.
	
	When $\|\vx\| = 1$, we have from \eqref{eq:n-cone-ball} that
	\begin{equation}\label{eq:dist-min}\dist\left(\vzero, \vA^\top(\vA\vx-\vb)+\cN_X(\vx)\right)=\min_{\lambda\ge 0}\|\vA^\top(\vA\vx-\vb)+\lambda \vx\|.
	\end{equation}
	If the minimizer of the right hand side of \eqref{eq:dist-min} is achieved at $\lambda = 0$, then \eqref{eq:reg-cond-poly} holds with $v=1$. Otherwise, the minimizer is $\lambda = -\vx^\top\vA^\top(\vA\vx-\vb)\ge0$. With this $\lambda$, we have
	\begin{align*}
	&~ \left[\dist\left(\vzero, \vA^\top(\vA\vx-\vb)+\cN_X(\vx)\right)\right]^2\\ 
	= &~ \|\vA^\top(\vA\vx-\vb)-\vx^\top\vA^\top(\vA\vx-\vb)\vx\|^2 \\
	= & ~(\vA\vx-\vb)^\top \vA(\vI-\vx\vx^\top)\vA^\top(\vA\vx-\vb).
	\end{align*}
	Let 
	\begin{align}\label{eq:v-min}
	v_*=\min_\vx \Big\{&\lambda_{\min}\left(\vA(\vI-\vx\vx^\top)\vA^\top\right), \nonumber\\
	&~\st~  \vx^\top\vA^\top(\vA\vx-\vb)\le0,\ \|\vx\|=1\Big\},
	\end{align}
	where $\lambda_{\min}(\cdot)$ denotes the minimum eigenvalue of a matrix. Then $v_*$ must be a finite nonnegative number. We show $v_*>0$. Otherwise suppose $v_*=0$, i.e., there is a $\vx$ such that $\vx^\top\vA^\top(\vA\vx-\vb)\le0$ and $\|\vx\|=1$, and also $\vA(\vI-\vx\vx^\top)\vA^\top$ is singular. Hence, there exists a $\vy\neq\vzero$ such that \begin{equation}\label{eq:sing-eq}
	\vA(\vI-\vx\vx^\top)\vA^\top\vy=\vzero.    
	\end{equation} 
	By scaling, we can assume $\|\vy\|=1$. Let $\vz = \vA^\top\vy$. Then $\|\vz\|=1$, and from \eqref{eq:sing-eq}, we have $\vz^\top(\vI-\vx\vx^\top)\vz = 1 - (\vz^\top\vx)^2 = 0$. 
	This equation implies $\vz = \vx$ or $\vz = -\vx$, because both $\vx$ and $\vz$ are unit vectors. Without loss of generality, we can assume $\vz = \vx$.
	Now recall $\vb=\vA\hat\vx$ with $\|\hat\vx\|<1$ and notice 
	\begin{align*}
	\vx^\top\vA^\top(\vA\vx-\vb) &= \vz^\top\vA^\top(\vA\vz-\vb) \\ 
	&= \vz^\top\vA^\top(\vy - \vA\hat\vx) = 1 - \vz^\top\vA^\top\vA\hat\vx > 0,
	\end{align*}
	where the inequality follows from $\|\vA\|=1$, $\|\vz\|=1$, and $\|\hat\vx\|<1$. Hence, we have a contradiction to $\vx^\top\vA^\top(\vA\vx-\vb)\le0$. Therefore, $v_*>0$.	
	
	Putting the above discussion together, we have that \eqref{eq:reg-cond-poly} holds with $v=\min\{1, v_*\}$, where $v_*$ is defined in \eqref{eq:v-min}. This completes the proof.

\subsection{Proof of Theorem~\ref{thm:ialm-compl}}
	First, note that $\cL_{\beta_k}(\cdot,\vy^k)$ is $\hat{L}_k$-smooth and $\hat{\rho}_k$-weakly convex, with $\hat{L}_k$ and $\hat{\rho}_k$ defined in \eqref{eq:hat_rho_k}. Then by the $\vx$ update in Algorithm~\ref{alg:ialm}, the stopping conditions of Algorithms~\ref{alg2} and \ref{alg:ippm}, and following the same proof of $\vareps$ stationarity as in Theorem~\ref{thm:ippm-compl}, we have
	\begin{equation}\label{temp1}
	\dist(\vzero, \partial_x \cL_{\beta_k}(\vx^{k+1},\vy^k)) \le \vareps, \forall k \ge 0.
	\end{equation}
	
	Next we give a uniform upper bound of the dual variable. By \eqref{eq:alm-y}, \eqref{eq:ialm_step}, $\vy^0 = \vzero$, and also the setting of $\gamma_k$, we have that $\forall k \ge 0$,
	\begin{align}
	\Vert \vy^k \Vert &\le \sum_{t=0}^{k-1} w_t \Vert \vc(\vx^{t+1}) \Vert \le \sum_{t=0}^{\infty} w_t \Vert \vc(\vx^{t+1}) \Vert \nonumber \\
	&\le \bar{c} w_0 \Vert \vc(\vx^0) \Vert (\log 2)^2 =: y_{\max}, \label{eq:y-bound}
	\end{align}
	where we have defined $\bar{c} = \sum_{t=0}^{\infty} \frac{1}{(t+1)^2 [\log (t+2)]^2}$.
	
	Combining the above bound with the regularity assumption \eqref{eq:regularity}, we have the following feasibility bound: for all $k \ge 1$,
	\begin{align}
	&\Vert \vc(\vx^k) \Vert \le \frac{1}{v\beta_{k-1}} \dist \left(0, \partial h(\vx^k) + \beta_{k-1} J_c(\vx^k)^{\top} \vc(\vx^k) \right) \nonumber \\
	= &\frac{1}{v\beta_{k-1}} \dist \big(0, \partial_x \cL_{\beta_k}(\vx^{k},\vy^{k-1}) - \nabla g(\vx^{k}) \nonumber \\ 
	&- J_c(\vx^k)^{\top} \vy^{k-1} \big) \nonumber \\
	\le &\frac{1}{v\beta_{k-1}} \big( \dist \left(0, \partial_x \cL_{\beta_k}(\vx^{k},\vy^{k-1}) \right) + \Vert \nabla g(\vx^k) \Vert \nonumber \\ 
	&+ \Vert J_c(\vx^k) \Vert \Vert \vy^{k-1} \Vert \big)  \nonumber \\
	\le &\frac{1}{v\beta_{k-1}} (\vareps + B_0 + B_c y_{\max}), \label{eq:c-bound}
	\end{align}
	where the third inequality follows from \eqref{temp1}, (\ref{P3-eqn-33}a), (\ref{P3-eqn-33}c), and \eqref{eq:y-bound}. 
	
	Now we define
	\begin{equation} \label{eq:K}
	K = \left\lceil \log_\sigma C_\vareps \right\rceil + 1, \text{ with } C_\vareps = \frac{\vareps + B_0 + B_c y_{\max}}{v \beta_0 \vareps}.
	\end{equation}
	Then by \eqref{eq:c-bound} and the setting of $\beta_k$ in Algorithm~\ref{alg:ialm}, we have $\Vert \vc(\vx^K) \Vert \le \vareps$. 
	Also recalling \eqref{temp1}, 
	we have
	\begin{equation*} 
	\dist(\vzero, \partial f_0(\vx^{k+1}) + J_c(\vx^{k+1})^{\top} \ (\vy^k + \beta_{k}\vc(\vx^{k+1}))) \le \vareps.
	\end{equation*}
	Therefore, $\vx^K$ is an $\vareps$-KKT point of \eqref{eq:ncp_eq} with the corresponding multiplier $\vy^{K-1} + \beta_{K-1} \vc(\vx^K)$, according to Definition \ref{def:eps-kkt}.
	
	In the rest of the proof, we bound the maximum number of iPPM iterations needed to stop Algorithm \ref{alg:ippm}, and the number of APG iterations per iPPM iteration needed to stop Algorithm \ref{alg2}, for each iALM outer iteration.
	
	Denote $\vx_k^t$ as the $t$-th iPPM iterate within the $k$-th outer iteration of iALM. Then at $\vx_k^t$,  
	we use APG to minimize $F_k^t(\cdot):=\cL_{\beta_k}(\cdot,\vy^k)+\hat{\rho}_k \Vert \cdot - \vx_k^t \Vert^2$, which is $\tilde{L}_k := (\hat{L}_k + 2 \hat{\rho}_k)$-smooth and $\hat{\rho}_k$-strongly convex. Hence, by Lemma \ref{lem:total-iter-apg}, at most $T_k^{\mathrm{APG}}$ (that is independent of $t$) APG iterations are required to find an $\frac{\vareps}{4}$ stationary point of $F_k^t(\cdot)$, where
	\begin{equation}\label{Tk-apg}
	T_k^{\mathrm{APG}} = \left\lceil \sqrt{\frac{\tilde{L}_k}{\hat{\rho}_k}}\log \frac{1024\tilde{L}_k^2(\tilde{L}_k+\hat{\rho}_k)D^2}{\vareps^2 \hat{\rho}_k} \right\rceil + 1, \forall k \ge 0.
	\end{equation}
	
	In addition, recalling the definition of $\cL_{\beta}$ in \eqref{eq:al-fun}, observe that for all $k\ge1$,
	\begin{align}
	\cL_{\beta_k}(\vx^k,\vy^k) \le & B_0 + \frac{\vareps+B_0+B_c y_{\max}}{v{\beta_0}} \nonumber \\
	&\left(y_{\max}+\frac{\sigma(\vareps + B_0 + B_c y_{\max})}{2v}\right)\sigma^{1-k} \label{eq:L-ub}\\
	\le & B_0 + \tilde{c}, \forall k \ge 1, \nonumber 
	\end{align}
	where $B_0$ is given in (\ref{P3-eqn-33}a) and\\ $\tilde{c} := \frac{\vareps+B_0+B_c y_{\max}}{v\beta_0} \left( y_{\max}+\frac{\sigma(\vareps+B_0+B_c y_{\max})}{2v} \right)$. Furthermore,
	\begin{align*}
	\cL_{\beta_0}(\vx^0,\vy^0) \le B_0 + \frac{\beta_0}{2} \Vert \vc(\vx^0) \Vert^2,
	\end{align*}
	and $\forall k \ge 0, \forall \vx \in \dom (h), $
	\begin{align}\label{eq:L-lb}
	\cL_{\beta_k}(\vx,\vy^k) \ge f_0(\vx) + \langle \vy^k,\vc(\vx) \rangle \ge -B_0-y_{\max} \bar{B}_c,
	\end{align}
	where $\bar{B}_c$ is given in (\ref{P3-eqn-33}c).
	
	Combining all three inequalities above with Theorem \ref{thm:ippm-compl} and $\hat{\rho}_k$-weak convexity of $\cL_{\beta_k}(\cdot,\vy^k)$, we conclude at most $T_k^{\mathrm{PPM}}$ iPPM iterations are needed to guarantee that $\vx^{k+1}$ is an $\vareps$ stationary point of $\cL_{\beta_k}(\cdot, \vy^k , \vz^k  )$, with
	\small{
		\begin{align}\label{Tk-ppm}
		T_k^{\mathrm{PPM}} &= \left\lceil \frac{32(\rho_0 + y_{\max}\bar{L}+\beta_k\rho_c)(2B_0+y_{\max}\bar{B}_c+\tilde{c})}{\vareps^2} \right\rceil, \forall k \ge 1\\ 
		T_0^{\mathrm{PPM}} &= \left\lceil\frac{32\rho_0}{\vareps^2}(2B_0+y_{\max}\bar{B}_c + \frac{\beta_0}{2} \Vert \vc(\vx^0) \Vert^2)\right\rceil.
		\end{align}
	}
	\normalsize
	
	Consequently, we have shown that at most $T$ total APG iterations are needed to find an $\vareps$-KKT point of \eqref{eq:ncp_eq}, where
	\begin{equation} \label{eq:T}
	T = \sum_{k=0}^{K-1} T_k^{\mathrm{PPM}} T_k^{\mathrm{APG}},
	\end{equation}
	with $K$ given in \eqref{eq:K}, $T_k^{\mathrm{APG}}$ given in \eqref{Tk-apg}, and $T_k^{\mathrm{PPM}}$ given in \eqref{Tk-ppm}. 
	
	The result in \eqref{eq:T} immediately gives us the following complexity results.
	
	By \eqref{eq:K}, we have $K = \tilde{O}(1)$ and $\beta_K = O(\vareps^{-1})$. Hence from \eqref{eq:hat_rho_k}, we have $\hat{\rho}_k = O(\beta_k),\hat{L}_k = O(\beta_k), \forall k \ge 0 $. Then by \eqref{Tk-apg}, $T_k^{\mathrm{APG}} = \tilde{O}(1), \forall k \ge 0$, and by \eqref{Tk-ppm}, we have 
	$T_k^{\mathrm{PPM}} = O(\vareps^{-3}), \forall k\ge0$. Therefore, in \eqref{eq:T}, $T = \tilde{O}(\vareps^{-3})$ for a general nonlinear $\vc(\cdot)$.
	
	For the special case when $\vc(\vx) = \vA \vx-\vb$, $\Vert \vc(\vx) \Vert^2$ is convex, so we have $\rho_c = 0$. Thus by \eqref{eq:hat_rho_k}, $\hat{\rho}_k = O(1), \forall k \ge 0 $. Hence in \eqref{Tk-ppm}, $T_k^{\mathrm{PPM}} = O(\vareps^{-2}), \forall k \ge 0$, and in \eqref{Tk-apg}, $T_k^{\mathrm{APG}} = \tilde{O}(\vareps^{-\frac{1}{2}}), \forall k \ge 0$. Therefore, by \eqref{eq:T}, $T = \tilde{O}(\vareps^{-\frac{5}{2}})$ for an affine $\vc(\cdot)$. This completes the proof.

\subsection{Proof of Theorem~\ref{thm:ialm-cplx-2}}
	First, by \eqref{eq:alm-y}, \eqref{wk-general} and $\vy^0 = \vzero$, we have
	\begin{align} \label{yk-bound-general}
	\Vert \vy^{k} \Vert &\le \sum_{t=0}^{k-1} w_t \Vert \vc(\vx^{t+1}) \Vert = \sum_{t=0}^{k-1} M(t+1)^q := y_k \nonumber\\ 
	&= O(k^{q+1}), \forall k \ge 0. 
	\end{align}
	Following the first part of the proof of Theorem \ref{thm:ialm-compl}, we can easily show that at most $K=O(\log \vareps^{-1})$ outer iALM iterations are needed to guarantee $\vx^{K}$ to be an $\vareps$-KKT point of \eqref{eq:ncp_eq}. Hence, $\beta_k = O(\vareps^{-1}), \forall\, 0 \le k \le K$.
	
	Combining the above bound on $K$ with \eqref{yk-bound-general}, we have 
	\begin{align*}
	\Vert \vy^k \Vert &\le y_{K} := \sum_{t=0}^{K-1} M(K+1)^q = O(K^{q+1})\\ 
	&= O\big((\log \vareps^{-1})^{q+1}\big),\, \forall 1 \le k \le K.
	\end{align*}
	Hence from \eqref{eq:hat_rho_k}, we have $\hat{\rho}_k = O(\beta_k) = O(\vareps^{-1}), \hat{L}_k = O(\beta_k) = O(\vareps^{-1}),\, \forall\, 0 \le k \le K$. 
	
	Notice that \eqref{eq:L-ub} and \eqref{eq:L-lb} still hold with $y_{\max}$ replaced by $y_k$. Hence,  $\forall k \le K, \forall \vx \in \dom(h),$
	\begin{equation*}
	\cL_{\beta_k}(\vx^k,\vy^k) - \cL_{\beta_k}(\vx,\vy^k) = O\left(y_k\left(1+\frac{y_k}{\beta_k}\right)\right).
	\end{equation*}
	The above equation together with Theorem \ref{thm:ippm-compl} gives that for any $k \le K$, at most $T_k^{\mathrm{PPM}}$ iPPM iterations are needed to terminate Algorithm \ref{alg:ippm} at the $k$-th outer iALM iteration, where
	\begin{align*}
	T_k^{\mathrm{PPM}} &= \left\lceil \frac{32\hat{\rho}_k}{\vareps^2}\big(\cL_{\beta_k}(\vx^k,\vy^k)-\min_{\vx} \cL_{\beta_k}(\vx,\vy^k)\big) \right\rceil\\ 
	&= O \left(\frac{\hat{\rho}_k y_k\left(1+\frac{y_k}{\beta_k}\right)}{\vareps^2} \right).
	\end{align*}
	Also, by Lemma \ref{lem:total-iter-apg}, at most $T_k^{\mathrm{APG}}$ APG iterations are needed to terminate Algorithm \ref{alg2}, where
	\begin{equation*}
	T_k^{\mathrm{APG}} = O\left(\sqrt{\frac{\hat{L}_k}{\hat{\rho_k}}}\log \vareps^{-1} \right), \forall k \ge 0.
	\end{equation*}
	
	Therefore, for all $k \le K$, 
	\begin{align*}
	T_k^{\mathrm{PPM}} T_k^{\mathrm{APG}} &= O\left( \frac{\sqrt{\hat{L}_k \hat{\rho}_k}\log \vareps^{-1}}{\vareps^2} y_k\left(1+\frac{y_k}{\beta_k}\right) \right) \\
	&= O\left( \frac{y_k \log \vareps^{-1}}{\vareps^2} (\beta_k + y_k) \right) \\
	&= O\left( \frac{k^{q+1} \log \vareps^{-1}}{\vareps^2} (\sigma^k + k^{q+1}) \right)\\
	& = O\left( \frac{K^{q+1} \log \vareps^{-1}}{\vareps^2} (\sigma^K + K^{q+1}) \right)\\
	&= O\left( \frac{(\log \vareps^{-1})^{q+2}}{\vareps^2} \left(\frac{1}{\vareps} + (\log \vareps^{-1})^{q+1}\right) \right)\\
	&= O\left( \frac{(\log \vareps^{-1})^{q+2}}{\vareps^3} \right),
	\end{align*}
	where the second equation is from $\hat{L}_k = O(\beta_k)$ and $\hat{\rho}_k = O(\beta_k)$ for a general nonlinear $\vc(\cdot)$, and the fifth one is obtained by $K = O(\log \vareps^{-1})$.
	
	Consequently, for a general nonlinear $\vc(\cdot)$, at most $T$ APG iterations in total are needed to find the $\vareps$-KKT point $\vx^K$, where
	\begin{equation*}
	T = \sum_{k=0}^{K-1} T_k^{\mathrm{PPM}} T_k^{\mathrm{APG}} = O\left(K \vareps^{-3} (\log \vareps^{-1})^{q+2} \right)  =\tilde{O}\left(\vareps^{-3} \right).
	\end{equation*}
	
	In the special case when $\vc(\vx) = \vA \vx-\vb$, the term $\Vert \vc(\vx) \Vert^2 = \Vert \vA \vx-\vb \Vert^2$ is convex, so we have $\rho_c = 0$. Hence, by \eqref{eq:hat_rho_k}, $\hat{\rho}_k = O(1), \forall k \ge 0 $. Then following the same arguments as above, we obtain that for any $k \le K$,
	\begin{align*}
	T_k^{\mathrm{PPM}} T_k^{\mathrm{APG}} &= O\left(\frac{\sqrt{\hat{L}_k \hat{\rho}_k}}{\vareps^{2}} (\log \vareps^{-1})^{q+2}\right)\\ 
	&= O\left(\vareps^{-\frac{5}{2}} (\log \vareps^{-1})^{q+2}\right).
	\end{align*}
	Therefore, at most $T$ total APG iterations are needed to find the $\vareps$-KKT point $\vx^K$, where
	\begin{equation*}
	T = \sum_{k=0}^{K-1} T_k^{\mathrm{PPM}} T_k^{\mathrm{APG}} = \tilde{O}\left(\vareps^{-\frac{5}{2}} \right),
	\end{equation*}
	which completes the proof.

\section{COMPLEXITY WITH CONVEX INEQUALITY CONSTRAINTS}\label{sec:cvx_ineq}
In this subsection, we consider problems with a nonconvex objective and convex constraints, formulated as 
\begin{equation}\label{eq:nc+c}
	\min_{\vx\in\RR^n} \big\{f_0(\vx) := g(\vx) + h(\vx), \st \vA\vx=\vb, \vf(\vx) \le \vzero \big\},
\end{equation}
where $g$ is continuously differentiable but possibly nonconvex, $h$ is closed convex but possibly nonsmooth, $\vA \in \RR^{l \times n}$, $\vb \in \RR^{l}$, and $\vf = (f_1, \ldots, f_m): \RR^n\to\RR^m$ with each $f_j$ convex and $L^f_j$-smooth. 
Denote $[\va]+$ as a vector taking component-wise positive part of $\va$ and $J_f(\vx)$ as the Jacobi matrix of $\vf$ at $\vx$. 
Assumption \ref{assumption:regularity} is generalized to cover inequality constraints:
\begin{assumption}[generalized regularity] \label{assumption:regularity_ineq}
	There is some $v > 0$ such that for any $k\ge1$, 
	\begin{equation}\label{eq:regularity_ineq}
		\textstyle v \sqrt{\Vert \vA\vx^k-\vb \Vert^2 + \Vert [\vf(\vx^k)]_+ \Vert^2} \le \dist \left(-\vA^{\top} (\vA\vx^k-\vb)-J_f(\vx^k)^{\top} [\vf(\vx^k)]_+, \frac{\partial h(\vx^k)}{\beta_{k-1}} \right).
	\end{equation}
\end{assumption}
The augmented Lagrangian function of \eqref{eq:nc+c} is (c.f. \citep{li2020augmented})
\begin{equation}\label{eq:AL_ineq}
	\cL_{\beta}(\vx,\vp) = f_0(\vx)+ \vy^{\top} (\vA\vx-\vb) + \frac{\beta}{2} \Vert \vA\vx-\vb \Vert^2 + \frac{1}{2\beta}(\Vert [\vz+\beta \vf(\vx)]_+ \Vert^2-\Vert \vz \Vert^2) ,
\end{equation} 
where $\vp = (\vy; \vz)$ is the multiplier vector.
Generalized from Definition \ref{def:eps-kkt}, $\vareps$-KKT point is defined as follows.
\begin{definition}[$\vareps$-KKT point]\label{def:eps-kkt-ineq}
	Given $\vareps \geq 0$, a point $\vx \in \RR^n$ is called an $\vareps$-KKT point to \eqref{eq:nc+c} if there exists $\vy\in\RR^l$ and $\vz \in \RR^m$ such that
	\begin{equation}\label{eq:kkt_ineq}
		\Vert \vA\vx-\vb \Vert \leq \vareps,\quad
		\dist\left(\vzero, \partial f_0(\vx)+\vA^\top  \vy + J_f^\top(\vx) \  \vz \right) \leq \vareps, \quad
		\sum_{i=1}^{m} | z_i f_i(\vx) | \le \vareps.
	\end{equation}
\end{definition}

Algorithm \ref{alg:ialm_ineq} is a direct modification of Algorithm \ref{alg:ialm} for the case of general convex constraints. 
\begin{algorithm} 
\caption{Inexact augmented Lagrangian method (iALM) for \eqref{eq:nc+c}}\label{alg:ialm_ineq}
\DontPrintSemicolon
\textbf{Initialization:} choose $\vx^0\in\dom(f_0), \vy^0 = \vzero$, $\vz^0 = \vzero$, $\beta_0>0$, $\sigma>1$\; 
\For{$k=0,1,\ldots,$}{
	Let $\beta_k=\beta_0\sigma^k$, $\phi(\cdot)=\cL_{\beta_k}(\cdot,p^k)-h(\cdot)$, and
	\begin{equation}\label{eq:hat_L_k}
\hat{L}_k = L_0 + \beta_k \Vert \vA^{\top}\vA \Vert + \sum_{i=1}^{m}(\beta_k B_i^f(B_i^f+L_i^f)+L_i^f|z^k_i|).
	\end{equation}\;
	Call Algorithm~\ref{alg:ippm} to obtain $\vx^{k+1} \leftarrow \mathrm{iPPM}(\phi,h,\vx^k, \rho_0, \hat{L}_k, \vareps)$\; 
	Update $\vy, \vz$ by
	\begin{align}
		\vy^{k+1} &= ~\vy^k  + w_k (\vA\vx^{k+1}-\vb),\label{eq:alm-y-ineq}\\
		z_i^{k+1} &= z_i^k + w_k \max \left\{ -\frac{z_i^k}{\beta_k},f_i(\vx^{k+1}) \right\}, \label{eq:alm-z-ineq}
	\end{align}
	where 
	\begin{align}\label{eq:ialm_step_ineq}
		\textstyle w_k &= w_0 \min\left\{1, \ \frac{\gamma_k }{\max\{\Vert \vA\vx^{k+1}-\vb \Vert, \Vert [\vf(\vx^{k+1})]_+ \Vert\} }, \ \frac{\beta_k}{w_0} \right\},\\
		\textstyle \gamma_k &= \frac{(\log 2)^2 \min \{ \Vert \vA\vx^{0}-\vb \Vert, \Vert [\vf(\vx^0)]_+ \Vert \} }{(k+1)[\log(k+2)]^2}.
	\end{align}
}
\end{algorithm}

Let $A_i$ be the $i$-th row of $\vA$. In this subsection, modified from \eqref{P3-eqn-33}, we denote 
\begin{subequations}\label{P3-eqn-33-ineq}
	\begin{align}
		&B_0=\max_{\vx\in\dom(h)}\max \big\{|f_0(\vx)|,\left\Vert \nabla g(\vx) \right\Vert\big\}, B_f = \max_{\vx\in\dom(h)}\Vert J_f(\vx) \Vert, \\
		&B_i=\max_{\vx\in\dom(h)}\max \big\{|A_i \vx - b_i|,\left\Vert A_i \right\Vert\big\}, \forall\, i \in [l], \\
		& B_i^f=\max_{\vx\in\dom(h)}\max \big\{|f_i(\vx)|,\left\Vert \nabla f_i(\vx) \right\Vert\big\}, \forall\, i \in [m], \\
		& \textstyle \bar{B}_c = \sqrt{\sum_{i=1}^l B_i^2},\quad \bar{B}_f = \sqrt{\sum_{i=1}^m (B_i^f)^2}, 
	\end{align}
\end{subequations}

We give the main convergence result of Algorithm \ref{alg:ialm_ineq} below.

\begin{theorem}[total complexity of iALM with convex constraints]\label{thm:ialm-compl-ineq}
	Suppose that all conditions in Assumptions~\ref{assump:smooth-wc}, \ref{assump:composite} and \ref{assumption:regularity_ineq} hold. Given $\vareps>0$, then Algorithm~\ref{alg:ialm_ineq} needs $\tilde{O}(\vareps^{-\frac{5}{2}})$ $\mathrm{APG}$ iterations to produce an $\vareps$-KKT solution of \eqref{eq:nc+c}. 
\end{theorem}

\begin{proof}
First, note that $\cL_{\beta_k}(\cdot,\vp^k)$ is $\hat{L}_k$-smooth (c.f. (31) in \citep{li2020augmented}) and $\rho_0$-weakly convex, with $\hat{L}_k$ defined in \eqref{eq:hat_L_k}. Then by the $\vx$-update in Algorithm~\ref{alg:ialm_ineq} and Theorem~\ref{thm:ippm-compl}, we have
\begin{equation}\label{temp1-ineq}
	\dist(\vzero, \partial_x \cL_{\beta_k}(\vx^{k+1},\vp^k)) \le \vareps, \forall k \ge 0.
\end{equation}

Next we give a uniform upper bound of the dual variable. By \eqref{eq:alm-y-ineq}, \eqref{eq:alm-z-ineq}, \eqref{eq:ialm_step}, $\vy^0 = \vzero$, $\vz^0 = \vzero$, and also the setting of $\gamma_k$, we have that $\forall k \ge 0$,
\begin{align}
	\Vert \vy^k \Vert &\le \sum_{t=0}^{k-1} w_t \Vert \vA\vx^{t+1}-\vb \Vert \le \sum_{t=0}^{\infty} w_t \Vert \vA\vx^{t+1}-\vb \Vert \le \bar{c} w_0 \Vert \vA\vx^{0}-\vb \Vert (\log 2)^2 =: y_{\max}, \label{eq:y-bound-ineq} \\
	\Vert \vz^k \Vert &\le \sum_{t=0}^{k-1} w_t \Vert [\vf(\vx^{t+1})]_+ \Vert \le \sum_{t=0}^{\infty} w_t \Vert [\vf(\vx^{t+1})]_+ \Vert \le \bar{c} w_0 \Vert [\vf(\vx^0)]_+ \Vert (\log 2)^2 =: z_{\max}, \label{eq:z-bound-ineq}
\end{align}
where we have defined $\bar{c} = \sum_{t=0}^{\infty} \frac{1}{(t+1)^2 [\log (t+2)]^2}$. 

Combining the above bounds with the regularity assumption \eqref{eq:regularity_ineq}, we have the following feasibility bound: for all $k \ge 1$, 
\begin{align}
	&\sqrt{\Vert \vA\vx^{k}-\vb \Vert^2 + \Vert [\vf(\vx^k)]_+ \Vert^2} \le \frac{1}{v\beta_{k-1}} \dist \Big( 0, \partial h(\vx^k) + \beta_{k-1} \vA^{\top} (\vA\vx^{k}-\vb) + \beta_{k-1} J_f(\vx^k)^{\top} [\vf(\vx^k)]_+ \Big) \nonumber \\
	= &\frac{1}{v\beta_{k-1}} \dist \Big(0, \partial_x \cL_{\beta_{k-1}}(\vx^{k},\vp^{k-1}) - \nabla g(\vx^{k}) \nonumber \\ 
	&- \vA^{\top} \vy^{k-1} - \sum_{i=1}^{m} \big( [z_i^{k-1}+\beta_{k-1}f_i(\vx^k)]_+ - \beta_{k-1}[f_i(\vx^k)]_+ \big)\nabla f_i(\vx^k) \Big) \nonumber \\
	\le &\frac{1}{v\beta_{k-1}} \Big( \dist \left(0, \partial_x \cL_{\beta_{k-1}}(\vx^{k},\vp^{k-1}) \right) + \Vert \nabla g(\vx^k) \Vert + \Vert \vA \Vert \Vert \vy^{k-1} \Vert + \Vert J_f(\vx^k) \Vert \Vert \vz^{k-1} \Vert \Big)  \nonumber \\
	\le &\frac{1}{v\beta_{k-1}} (\vareps + B_0 + \Vert \vA \Vert y_{\max} + B_f z_{\max}), \label{eq:c-bound-ineq}
\end{align}
where the third inequality follows from
(\ref{P3-eqn-33-ineq}a), \eqref{temp1-ineq},  \eqref{eq:y-bound-ineq}, and \eqref{eq:z-bound-ineq}. 

Now we define
\begin{equation} \label{eq:K-ineq} 
	K = \left\lceil \log_\sigma \hat{C}_\vareps \right\rceil + 1, \text{ with } \hat{C}_\vareps = \max \left\{ C_{\vareps}, 3z_{\max}C_{\vareps}, 3\beta_0 \vareps C_{\vareps}^2, \frac{3z_{\max}^2}{\beta_0 \vareps} \right\}, \text{ where } C_\vareps = \frac{\vareps + B_0 + \Vert \vA \Vert y_{\max} + B_f z_{\max}}{v \beta_0 \vareps}.
\end{equation}
Then by \eqref{eq:c-bound-ineq}, \eqref{eq:K-ineq}, and the setting of $\beta_k$ in Algorithm~\ref{alg:ialm_ineq}, we have 
\begin{equation} \label{eq:pres-ineq}
	\sqrt{\Vert \vA\vx^{K}-\vb \Vert^2 + \Vert [\vf(\vx^K)]_+ \Vert^2} \le \vareps. 
\end{equation}
Also recalling \eqref{temp1-ineq}, 
we have
\begin{equation} \label{eq:dres-ineq} 
	\dist\Big(\vzero, \partial f_0(\vx^{K}) + \vA^{\top} \ (\vy^{K-1} + \beta_{K-1} (\vA\vx^{K}-\vb) ) + J_f(\vx^{K}) \ [\vz^{K-1}+\beta_{K-1}\vf(\vx^K)]_+ \Big) \le \vareps.
\end{equation}

By \eqref{eq:alm-z-ineq} and \eqref{eq:ialm_step_ineq}, we have $w_k \le \beta_k$ and $\vz^k \ge \vzero$, $\forall k \ge 0$. Hence, for all $k \ge 0$, 
\begin{align*}
&\sum_{i=1}^m \Big| [z_i^k+\beta_k f_i(\vx^{k+1})]_+ f_i(\vx^{k+1}) \Big|\\ 
= &\sum_{i:z_i^k+\beta_k f_i(\vx^{k+1})>0} (z_i^k+\beta_k f_i(\vx^{k+1})) \big| f_i(\vx^{k+1}) \big|\\
\le &\sum_{i:z_i^k+\beta_k f_i(\vx^{k+1})>0, f_i(\vx^{k+1})>0} (z_i^k+\beta_k f_i(\vx^{k+1})) [f_i(\vx^{k+1})]_+  + \sum_{i: -\frac{z_i^k}{\beta_k}<f_i(\vx^{k+1})<0} z_i^k \cdot \frac{z_i^k}{\beta_k}\\
\le & \sum_{i=1}^{m} z_i^k [f_i(\vx^{k+1})]_+ + \sum_{i=1}^{m}\beta_k [f_i(\vx^{k+1})]_+^2 + \sum_{i=1}^{m} \frac{(z_i^k)^2}{\beta_k}\\
\le & \Vert \vz^k \Vert \cdot \Vert [\vf(\vx^{k+1})]_+ \Vert + \beta_k \Vert [\vf(\vx^{k+1})]_+ \Vert^2 + \frac{\Vert \vz^k \Vert^2}{\beta_k} \\
\le & z_{\max} \Vert [\vf(\vx^{k+1})]_+ \Vert + \beta_k \Vert [\vf(\vx^{k+1})]_+ \Vert^2 + \frac{z_{\max}^2}{\beta_k}
\end{align*}
By \eqref{eq:c-bound-ineq} and \eqref{eq:K-ineq}, 
\begin{align*}
& z_{\max} \Vert [\vf(\vx^{K})]_+ \Vert + \beta_{K-1} \Vert [\vf(\vx^{K})]_+ \Vert^2 + \frac{z_{\max}^2}{\beta_{K-1}}\\ 
\le & \frac{z_{\max}}{v\beta_{K-1}} (\vareps + B_0 + \Vert \vA \Vert y_{\max} + B_f z_{\max}) + \frac{1}{v^2\beta_{K-1}} (\vareps + B_0 + \Vert \vA \Vert y_{\max} + B_f z_{\max})^2 + \frac{z_{\max}^2}{\beta_{K-1}} \\
\le & \frac{\vareps}{3} + \frac{\vareps}{3} + \frac{\vareps}{3} = \vareps.
\end{align*}
Hence, 
\begin{equation} \label{eq:cs-ineq}
	\sum_{i=1}^m \Big| [z_i^{K-1}+\beta_{K-1} f_i(\vx^{K})]_+ f_i(\vx^{K}) \Big| \le \vareps.
\end{equation}

Therefore by \eqref{eq:pres-ineq}, \eqref{eq:dres-ineq} and \eqref{eq:cs-ineq},  $\vx^K$ is an $\vareps$-KKT point of \eqref{eq:nc+c} with the corresponding multiplier $\big(\vy^{K-1} + \beta_{K-1} \vc(\vx^K), [\vz^{K-1}+\beta_{K-1}\vf(\vx^K)]_+ \big)$, according to Definition \ref{def:eps-kkt-ineq}.

In the rest of the proof, we bound the maximum number of iPPM iterations needed to stop Algorithm \ref{alg:ippm}, and the number of APG iterations per iPPM iteration needed to stop Algorithm \ref{alg2}, for each iALM outer iteration. 

Denote $\vx_k^t$ as the $t$-th iPPM iterate within the $k$-th outer iteration of iALM. Then at $\vx_k^t$,  
we use APG to minimize $F_k^t(\cdot):=\cL_{\beta_k}(\cdot,\vp^k)+\rho_0 \Vert \cdot - \vx_k^t \Vert^2$, which is $\tilde{L}_k := (\hat{L}_k + 2 \rho_0)$-smooth and $\rho_0$-strongly convex. Hence, by Lemma \ref{lem:total-iter-apg}, at most $T_k^{\mathrm{APG}}$ (that is independent of $t$) APG iterations are required to find an $\frac{\vareps}{4}$ stationary point of $F_k^t(\cdot)$, where
\begin{equation}\label{Tk-apg-ineq}
	T_k^{\mathrm{APG}} = \left\lceil \sqrt{\frac{\tilde{L}_k}{\rho_0}}\log \frac{1024\tilde{L}_k^2(\tilde{L}_k+\rho_0)D^2}{\vareps^2 \rho_0} \right\rceil + 1, \forall k \ge 0.
\end{equation}

In addition, recalling the definition of $\cL_{\beta}$ in \eqref{eq:AL_ineq} and by \eqref{eq:c-bound-ineq}, we have for all $k\ge1$,
\begin{align} 
	\cL_{\beta_k}(\vx^k,\vp^k) \le & B_0 + \frac{\vareps + B_0 + \Vert \vA \Vert y_{\max} + B_f z_{\max} }{v \beta_0 } \left(y_{\max}+\frac{\sigma( \vareps + B_0 + \Vert \vA \Vert y_{\max} + B_f z_{\max} )}{2v}\right)\sigma^{1-k} \nonumber \\
	& + z_{\max}\Vert [\vf(\vx^k)]_+ \Vert + \frac{\beta_k}{2}\Vert [\vf(\vx^k)]_+ \Vert^2 \nonumber \\
	\le & B_0 + \frac{\vareps + B_0 + \Vert \vA \Vert y_{\max} + B_f z_{\max} }{v \beta_0 } \left(y_{\max}+z_{\max}+\frac{\sigma( \vareps + B_0 + \Vert \vA \Vert y_{\max} + B_f z_{\max} )}{v}\right)\sigma^{1-k} \label{eq:L-ub-ineq} \\
	\le & B_0 + \tilde{c}, \forall k \ge 1, \nonumber 
\end{align}
where $B_0$ is given in (\ref{P3-eqn-33-ineq}a) and
$$\tilde{c} := \frac{\vareps+B_0+\Vert \vA \Vert y_{\max} + B_f z_{\max}}{v\beta_0} \left( y_{\max}+z_{\max}+\frac{\sigma(\vareps + B_0 + \Vert \vA \Vert y_{\max} + B_f z_{\max})}{v} \right).$$
Furthermore, 
\begin{align*}
	\cL_{\beta_0}(\vx^0,\vp^0) \le B_0 + \frac{\beta_0}{2} (\Vert \vA\vx^{0}-\vb \Vert^2+\Vert [\vf(\vx^0)]_+ \Vert^2),
\end{align*}
and $\forall k \ge 0, \forall \vx \in \dom (h), $
\begin{align}\label{eq:L-lb-ineq}
	\cL_{\beta_k}(\vx,\vp^k) \ge f_0(\vx) + \langle \vy^k,\vA\vx-\vb \rangle - \frac{\Vert \vz^k \Vert^2}{2\beta_k} \ge -B_0-y_{\max} \bar{B}_c - \frac{z_{\max}^2}{2\beta_0 \sigma^k},
\end{align}
where $\bar{B}_c$ is given in (\ref{P3-eqn-33-ineq}d).

Combining all three inequalities above with Theorem \ref{thm:ippm-compl} and $\rho_0$-weak convexity of $\cL_{\beta_k}(\cdot,\vp^k)$, we conclude at most $T_k^{\mathrm{PPM}}$ iPPM iterations are needed to guarantee that $\vx^{k+1}$ is an $\vareps$ stationary point of $\cL_{\beta_k}(\cdot, \vp^k  )$, with 
\small{
	\begin{align}\label{Tk-ppm-ineq}
		T_k^{\mathrm{PPM}} &= \left\lceil \frac{32 \rho_0 }{\vareps^2} \left( 2B_0+y_{\max}\bar{B}_c +\frac{z_{\max}^2}{2\beta_0 \sigma^k} +\tilde{c} \right) \right\rceil, \forall k \ge 1\\ 
		T_0^{\mathrm{PPM}} &= \left\lceil\frac{32\rho_0}{\vareps^2} \left( 2B_0+y_{\max}\bar{B}_c +\frac{z_{\max}^2}{2\beta_0 \sigma^k} + \frac{\beta_0}{2} (\Vert \vA\vx^0-\vb \Vert^2+\Vert [\vf(\vx^0)]_+ \Vert^2) \right)\right\rceil. \label{T0-ppm-ineq}
	\end{align}
}
\normalsize

Consequently, we have shown that at most $T$ total APG iterations are needed to find an $\vareps$-KKT point of \eqref{eq:nc+c}, where
\begin{equation} \label{eq:T-ineq}
	T = \sum_{k=0}^{K-1} T_k^{\mathrm{PPM}} T_k^{\mathrm{APG}},
\end{equation}
with $K$ given in \eqref{eq:K-ineq}, $T_k^{\mathrm{APG}}$ given in \eqref{Tk-apg-ineq}, and $T_k^{\mathrm{PPM}}$ given in \eqref{Tk-ppm-ineq} and \eqref{T0-ppm-ineq}. 

The result in \eqref{eq:T-ineq} immediately gives us the following complexity results.

By \eqref{eq:K-ineq}, we have $K = \tilde{O}(1)$ and $\beta_K = O(\vareps^{-1})$. Hence from \eqref{eq:hat_L_k}, we have $\hat{L}_k = O(\vareps^{-1}), \forall k \ge 0 $. Then by \eqref{Tk-apg-ineq}, $T_k^{\mathrm{APG}} = \tilde{O}(\vareps^{-\frac{1}{2}}), \forall k \ge 0$, and by \eqref{Tk-ppm-ineq} and \eqref{T0-ppm-ineq}, we have 
$T_k^{\mathrm{PPM}} = O(\vareps^{-2}), \forall k\ge0$. Therefore, in \eqref{eq:T-ineq}, $T = \tilde{O}(\vareps^{-\frac{5}{2}})$. This completes the proof.
\end{proof}


\section{ADDITIONAL TABLES}
We provide more detailed experimental results on the LCQP and EV problems to demonstrate the empirical performance of the proposed iALM from another perspective. We compare our method with the iALM in~\citep{sahin2019inexact} on LCQP and EV, and the HiAPeM in~\citep{li2020augmented} on LCQP. 

For each method, we report the primal residual, dual residual, running time (in seconds), and the number of gradient evaluation, shortened as \verb|pres|, \verb|dres|, \verb|time|, and \verb|#Grad|, respectively. The results for all trials are shown in Tables \ref{table:qp-small-1} and \ref{table:qp-large-1} for the LCQP problem, and in Tables \ref{table:ev-small} and \ref{table:ev-large} for the EV problem. From the results, we conclude that for both of the LCQP and EV problems, to reach the same-accurate KKT point of each tested instance, the proposed improved iALM needs significantly fewer gradient evaluations and takes far less time than all other compared methods.

\setlength{\tabcolsep}{3pt}

\begin{table}[h]\caption{Results by the proposed improved iALM, the iALM by \cite{sahin2019inexact}, and the HiAPeM by \cite{li2020augmented} on solving a $1$-weakly convex LCQP \eqref{eq:ncQP} of size $m=10$ and $n=200$. }\label{table:qp-small-1} 
	\begin{center}
		\resizebox{1 \textwidth}{!}{ 
			\begin{tabular}{|c||cccc|cccc|cccc|cccc|} 
				\hline
				trial & pres & dres & time & \#Grad & pres & dres & time & \#Grad & pres & dres & time & \#Grad & pres & dres & time & \#Grad \\\hline\hline 
				&\multicolumn{4}{|c|}{proposed improved iALM} & \multicolumn{4}{|c|}{iALM by \cite{sahin2019inexact}} &\multicolumn{4}{|c|}{HiAPeM with $N_0 = 10, N_1 = 2$} &\multicolumn{4}{|c|}{HiAPeM with $N_0 = 1, N_1 = 10^6$} \\\hline
				1 	 & 2.29e-4 & 8.31e-4 & 2.09 & 47468 & 7.06e-4 & 1.00e-3 & 15.56 & 1569788  
				& 3.77e-5 & 9.64e-4 & 2.61 & 150653 & 2.28e-4 & 7.25e-4 & 3.93 & 323020\\ 
				2 	 & 1.94e-4 & 9.24e-4 & 1.00 & 26107 & 1.94e-4 & 1.00e-3 & 6.68 & 713807 
				& 4.02e-4 & 6.45e-4 & 2.51 & 154519 & 3.72e-4 & 4.83e-4 & 6.23 & 531680\\
				3 	 & 2.23e-4 & 3.29e-4 & 1.35 & 33392 & 1.40e-4 & 1.00e-3 & 5.37 & 636043 
				& 7.16e-5 & 6.37e-4 & 2.06 & 135379 & 3.41e-4 & 9.35e-4 & 5.54 & 458308\\
				4 	 & 6.58e-4 & 7.18e-4 & 2.21 & 41325 & 6.58e-4 & 1.00e-3 & 9.39 & 1048446 
				& 1.33e-4 & 8.29e-4 & 1.53 & 82087 & 3.49e-4 & 7.10e-4 & 4.67 & 389567\\
				5 	 & 2.22e-4 & 5.43e-4 & 1.04 & 29252 & 1.80e-4 & 1.00e-3 & 9.56 & 1100625 
				& 1.46e-4 & 4.60e-4 & 3.11 & 216479 & 2.95e-4 & 9.21e-4 & 8.97 & 735546\\
				6 	 & 1.75e-4 & 5.04e-4 & 1.25 & 34488 & 8.96e-4 & 1.00e-3 & 11.03 & 1339160 
				& 9.82e-5 & 7.36e-4 & 0.64 & 31099 & 3.35e-4 & 7.94e-4 & 3.32 & 272395\\
				7 	 & 4.03e-4 & 5.04e-4 & 1.10 & 28636 & 1.98e-4 & 1.00e-3 & 7.97 & 927075 
				& 3.00e-4 & 7.38e-4 & 3.00 & 199126 & 3.89e-4 & 8.39e-4 & 6.69 & 544974\\
				8 	 & 5.83e-4 & 4.58e-4 & 1.70 & 39719 & 8.62e-4 & 1.00e-3 & 8.77 & 982164 
				& 3.93e-4 & 7.13e-4 & 2.85 & 189818 & 4.62e-4 & 9.09e-4 & 4.18 & 338027\\
				9 	 & 5.98e-4 & 3.70e-4 & 1.66 & 37379 & 5.98e-4 & 1.00e-3 & 5.23 & 560382 
				& 1.45e-4 & 9.63e-4 & 4.34 & 286666 & 2.80e-4 & 9.45e-4 & 9.78 & 751636\\ 
				10 	 & 8.11e-4 & 3.07e-4 & 1.05 & 25170 & 8.23e-4 & 1.00e-3 & 30.75 & 3474626 
				& 2.45e-4 & 8.45e-4 & 4.49 & 278127 & 4.65e-4 & 9.30e-4 & 7.47 & 594326\\\hline
				avg. & 4.10e-4 & 5.49e-4 & 1.44 & 34294 & 5.26e-4 & 1.00e-3 & 11.03 & 1235210 
				& 1.97e-4 & 7.53e-4 & 2.71 & 172395 & 3.52e-4 & 8.20e-4 & 6.08 & 493948 \\\hline
			\end{tabular}
		}
	\end{center}
\end{table}

\begin{table}[h]\caption{Results by the proposed improved iALM, the iALM by \cite{sahin2019inexact}, and the HiAPeM by \cite{li2020augmented} on solving a 1-weakly convex LCQP \eqref{eq:ncQP} of size $m=100$ and $n=1000$.}\label{table:qp-large-1} 
	\begin{center}
		\resizebox{1 \textwidth}{!}{
			\begin{tabular}{|c||cccc|cccc|cccc|cccc|} 
				\hline
				trial & pres & dres & time & \#Grad & pres & dres & time & \#Grad & pres & dres & time & \#Grad & pres & dres & time & \#Grad \\\hline\hline 
				&\multicolumn{4}{|c|}{proposed improved iALM} & \multicolumn{4}{|c|}{iALM by \cite{sahin2019inexact}} &\multicolumn{4}{|c|}{HiAPeM with $N_0 = 10, N_1 = 2$} &\multicolumn{4}{|c|}{HiAPeM with $N_0 = 1, N_1 = 10^6$}\\\hline
				1 	 & 4.36e-4 & 8.65e-4 & 109.90 & 220937 & 5.80e-4 & 8.1e-3 & 2281.8 & 13098032  
				& 1.05e-4 & 9.96e-4 & 550.18 & 2823733 & 5.35e-4 & 8.24e-4 & 897.68 & 5228014\\
				2 	 & 4.07e-4 & 7.47e-4 & 144.23 & 280500 & 5.90e-4 & 1.1e-3 & 1682.5 & 10207308 
				& 1.67e-4 & 9.04e-4 & 597.60 & 2879969 & 5.51e-4 & 8.05e-4 & 740.28 & 4540532\\
				3 	 & 5.99e-4 & 9.70e-4 & 99.37 & 228324 & 8.73e-4 & 1.00e-3 & 1281.3 & 8587300 
				& 8.22e-4 & 6.92e-4 & 474.76 & 2697241 & 5.67e-4 & 9.97e-4 & 1314.3 & 6986241\\
				4 	 & 4.59e-4 & 8.53e-4 & 179.91 & 311724 & 4.05e-4 & 2.1e-3 & 1548.6 & 8474538 
				& 4.10e-5 & 8.20e-4 & 747.18 & 3804152 & 5.16e-4 & 8.62e-4 & 741.43 & 4281876\\
				5 	 & 6.69e-4 & 9.57e-4 & 162.06 & 367321 & 3.96e-4 & 1.33e-2 & 1802.0 & 12464010 
				& 1.17e-4 & 9.82e-4 & 603.44 & 3008964 & 5.16e-4 & 9.11e-4 & 667.01 & 3830799\\
				6 	 & 6.85e-4 & 8.84e-4 & 104.30 & 200256 & 1.49e-4 & 1.6e-3 & 2010.8 & 13071595 
				& 5.16e-4 & 9.11e-4 & 667.01 & 3830799 & 5.79e-4 & 9.82e-4 & 1396.0 & 8174370\\
				7 	 & 6.10e-4 & 9.30e-4 & 124.50 & 244074 & 4.56e-4 & 1.4e-3 & 1843.8 & 11843900 
				& 4.78e-4 & 7.73e-4 & 712.36 & 3658514 & 5.53e-4 & 9.25e-4 & 615.96 & 3609496\\
				8 	 & 8.47e-4 & 7.40e-4 & 122.57 & 261206 & 4.81e-4 & 2.3e-3 & 1520.6 & 10298480 
				& 7.69e-4 & 6.36e-4 & 402.49 & 2036351 & 5.47e-4 & 9.78e-4 & 520.07 & 2681970\\
				9 	 & 5.16e-4 & 8.91e-4 & 165.14 & 316827 & 2.08e-4 & 1.3e-3 & 2334.9 & 14446205 
				& 5.08e-4 & 4.83e-4 & 561.30 & 3268825 & 5.43e-4 & 8.26e-4 & 1059.6 & 6958198\\
				10 	 & 3.46e-4 & 9.72e-4 & 142.67 & 352781 & 3.13e-4 & 1.5e-3 & 1519.9 & 9370342 
				& 8.36e-5 & 9.60e-4 & 542.09 & 2807758 & 5.54e-4 & 8.98e-4 & 1963.1 & 11091867\\\hline
				avg. & 5.57e-4 & 8.81e-4 & 135.47 & 278395 & 4.45e-4 & 3.37e-3 & 1782.6 & 11186171 
				& 3.61e-4 & 8.16e-4 & 585.84 & 3081631 & 5.46e-4 & 9.01e-4 & 991.54 & 5738336 \\\hline
			\end{tabular}
		}
	\end{center}
\end{table}

\setlength{\tabcolsep}{5pt}

\begin{table}[h]\caption{Results by the proposed improved iALM and the iALM by \cite{sahin2019inexact} on solving a generalized eigenvalue problem \eqref{eq:EV} of size $n=200$. }\label{table:ev-small} 
	\begin{center}
		{\small
			\begin{tabular}{|c||ccccc|cccc|} 
				\hline
				trial & pres & dres & time & \#Obj & \#Grad & pres & dres & time & \#Grad\\\hline\hline 
				&\multicolumn{5}{|c|}{proposed improved iALM} & \multicolumn{4}{|c|}{iALM by \cite{sahin2019inexact}}  \\\hline
				1 	 & 1.39e-4 & 9.98e-4 & 1.09 & 46140 & 38245 & 1.39e-4 & 1.00e-3 & 2.84 & 233367  
				\\ 
				2 	 & 5.69e-4 & 9.87e-4 & 0.48 & 31456 & 25592 & 5.69e-4 & 1.00e-3 & 1.32 & 144750 
				\\
				3 	 & 2.57e-4 & 9.92e-4 & 0.60 & 32933 & 26112 & 2.57e-4 & 1.00e-3 & 2.21 & 150136 
				\\
				4 	 & 1.45e-4 & 9.98e-4 & 0.59 & 29408 & 25203 & 1.45e-4 & 1.00e-3 & 2.24 & 153485 
				\\
				5 	 & 1.52e-4 & 1.00e-3 & 0.93 & 37477 & 27434 & 1.51e-4 & 1.00e-3 & 1.63 & 153596 
				\\
				6 	 & 2.34e-4 & 9.71e-4 & 0.29 & 17765 & 14353 & 2.34e-4 & 1.00e-3 & 0.59 & 60643 
				\\
				7 	 & 9.06e-4 & 9.98e-4 & 0.42 & 26032 & 20886 & 9.06e-4 & 1.00e-3 & 1.05 & 109958 
				\\
				8 	 & 6.57e-4 & 9.97e-4 & 0.42 & 24184 & 19974 & 6.57e-4 & 1.00e-3 & 1.53 & 104508 
				\\
				9 	 & 2.44e-4 & 9.95e-4 & 0.45 & 27125 & 22390 & 2.44e-4 & 1.00e-3 & 1.20 & 126874 
				\\ 
				10 	 & 2.16e-4 & 9.98e-4 & 0.49 & 31238 & 26527 & 2.16e-4 & 1.00e-3 & 1.55 & 160941 
				\\\hline
				avg. & 3.52e-4 & 9.03e-4 & 0.58 & 30376 & 24672 & 3.52e-4 & 1.00e-3 & 1.62 & 139823 
				\\\hline
			\end{tabular}
		}
	\end{center}
\end{table}

\begin{table}[h]\caption{Results by the proposed improved iALM and the iALM by \cite{sahin2019inexact} on solving a generalized eigenvalue problem \eqref{eq:EV} of size $n=1000$. }\label{table:ev-large} 
	\begin{center}
		{\small
			\begin{tabular}{|c||ccccc|cccc|} 
				\hline
				trial & pres & dres & time & \#Obj & \#Grad & pres & dres & time & \#Grad\\\hline\hline 
				&\multicolumn{5}{|c|}{proposed improved iALM} & \multicolumn{4}{|c|}{iALM by \cite{sahin2019inexact}}  \\\hline
				1 	 & 6.87e-4 & 9.78e-4 & 60.77 & 56805 & 42626 & 6.86e-4 & 2.5e-3 & 5671.9 & 9329514  
				\\ 
				2 	 & 1.39e-4 & 9.85e-4 & 63.29 & 80454 & 60765 & 1.38e-4 & 4.3e-3 & 8128.5 & 13295555 
				\\
				3 	 & 5.94e-4 & 9.92e-4 & 60.87 & 70884 & 49616 & 5.94e-4 & 1.00e-3 & 5070.0 & 8585272
				\\
				4 	 & 4.20e-4 & 9.97e-4 & 51.08 & 73494 & 51707 & 4.20e-4 & 1.00e-3 & 6045.3 & 10008459 
				\\
				5 	 & 6.27e-4 & 9.99e-4 & 65.20 & 72763 & 52095 & 6.27e-4 & 1.6e-3 & 6733.4 & 10820619 
				\\
				6 	 & 2.92e-4 & 9.82e-4 & 36.16 & 41402 & 32164 & 2.90e-4 & 3.1e-3 & 3936.9 & 6588034 
				\\
				7 	 & 3.35e-4 & 9.95e-4 & 87.89 & 104069 & 74808 & 3.35e-4 & 2.1e-3 & 9183.8 & 15689148 
				\\
				8 	 & 4.47e-4 & 9.91e-4 & 51.12 & 60555 & 45578 & 4.46e-4 & 2.6e-3 & 5300.0 & 9039022 
				\\
				9 	 & 4.02e-4 & 9.91e-4 & 44.23 & 51399 & 39064 & 4.01e-4 & 2.6e-3 & 4771.7 & 8466906 
				\\ 
				10 	 & 9.32e-4 & 9.95e-4 & 79.42 & 98130 & 69322 & 9.32e-4 & 1.6e-3 & 8846.8 & 14688990 
				\\\hline
				avg. & 4.88e-4 & 9.91e-4 & 60.00 & 70996 & 51775 & 4.87e-4 & 2.24e-3 & 5975.1 & 10651152 
				\\\hline
			\end{tabular}
		}
	\end{center}
\end{table}

In Table \ref{table:qp_add} below, we also compare our proposed iALM with the iPPP method in \citep{lin2019inexact-pp} on one representative instance of the LCQP problem in Section \ref{sec:experiment-ncvx}. For iPPP, we tune $\beta_k = \beta_0 \cdot k$ with $\beta_0 = 10$.

\begin{table}[htbp]\caption{Results by the proposed improved iALM and the iPPP by \cite{lin2019inexact-pp} on solving an LCQP problem \eqref{eq:EV} of size $m = 100$ and $n = 1000$. }\label{table:qp_add} 
	\begin{center}
		{\small
			\begin{tabular}{|c||cccc|} 
	\hline
	method & pres & dres & time & \#Grad 
	\\\hline\hline 
	proposed iALM & 4.08e-4 & 7.47e-4 & 293.2 & 280500 \\ iPPP in (Lin et al., 2019) & 9.98e-4 & 9.98e-4 & 930.7 & 1644496 \\\hline
\end{tabular}
		}
	\end{center}
\end{table}

\FloatBarrier 

\section{BETTER SUBROUTINE BY INEXACT PROXIMAL POINT METHOD}
We mentioned at the end of Section \ref{sec:alg} that our iPPM is more stable and more efficient on solving nonconvex subproblems in the form of \eqref{eq:comp-prob} than the subroutine by \cite{sahin2019inexact}. An intuitive explanation is as follows. The iPPM tackles the nonconvex problem by solving a sequence of perturbed strongly convex problems, which can be solved by Nesterov's accelerated first-order method. 
In contrast, the subroutine of the iALM by \cite{sahin2019inexact} applies Nesterov's acceleration technique directly while performing proximal gradient update to solve the nonconvex problem. 
We believe such a combination of acceleration with nonconvexity attributes to the instability or inefficiency of the iALM by \cite{sahin2019inexact}.

In this section, we provide numerical results to support the claim above. In Figure \ref{fig:dres} below, we plot representative trajectories of the violation of stationarity for the first subproblem in all of our experiments (namely, LCQP, EV and clustering problems) using our iPPM and the subsolver by \cite{sahin2019inexact} started from the same initial points, where the violation of stationarity is measured as $\dist\big(\vzero,\partial F(\vx)\big)$. 
From the figure, we can clearly observe that our iPPM method is more efficient than the subsolver by \cite{sahin2019inexact}. 

\begin{figure}[h] 
	\begin{center}
		\resizebox{1 \textwidth}{!}{
			\begin{tabular}{ccc}
				instance of LCQP \eqref{eq:ncQP} subproblem & instance of EV \eqref{eq:EV} subproblem & clustering \eqref{eq:cluster} subproblem with Iris data \\
				size $m = 20$ and $n = 100$ & size $n = 200$ & size $(n,s,r) = (150,100,6)$ \\[-0.1cm]
				\includegraphics[width=0.35\textwidth]{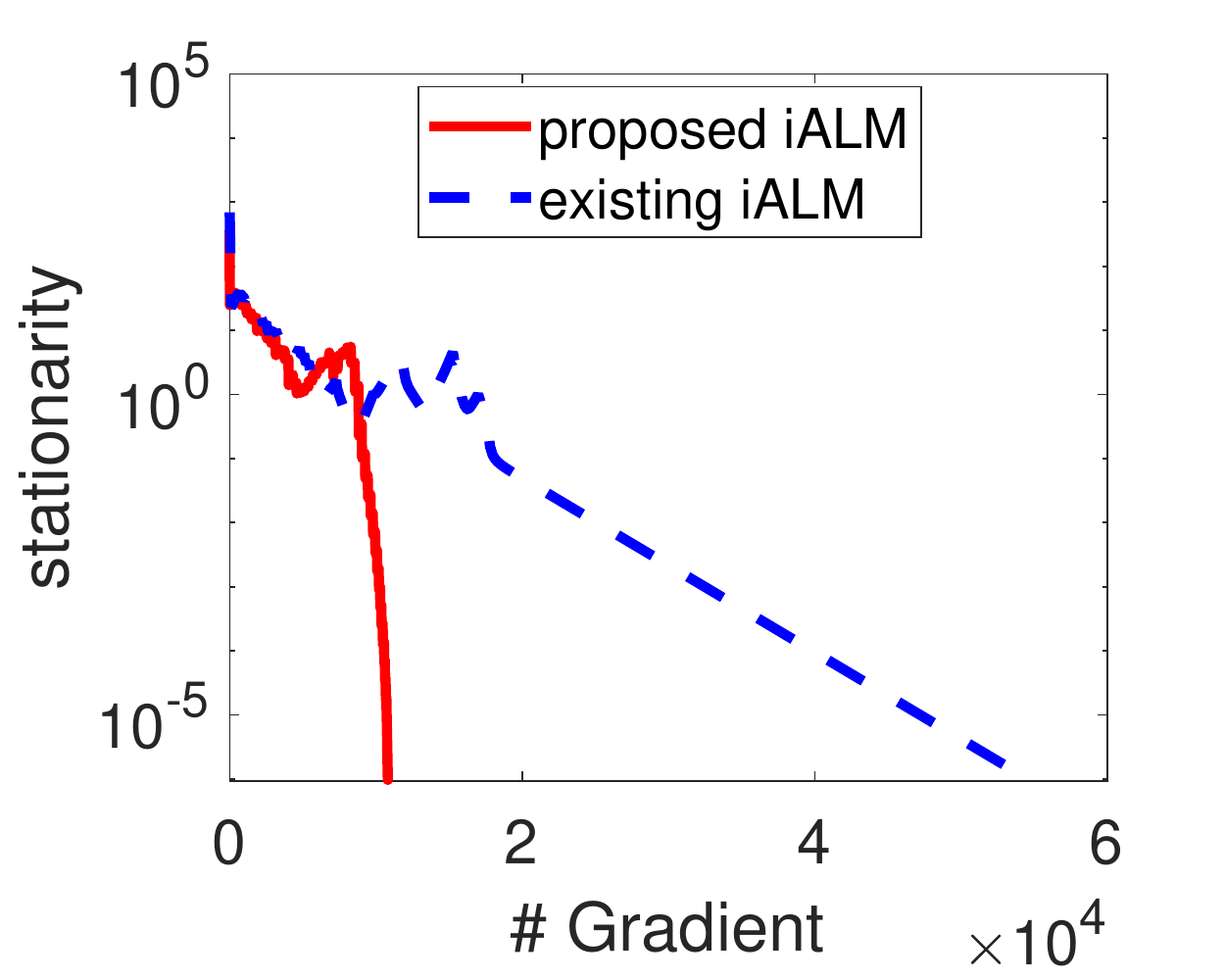} & \includegraphics[width=0.35\textwidth]{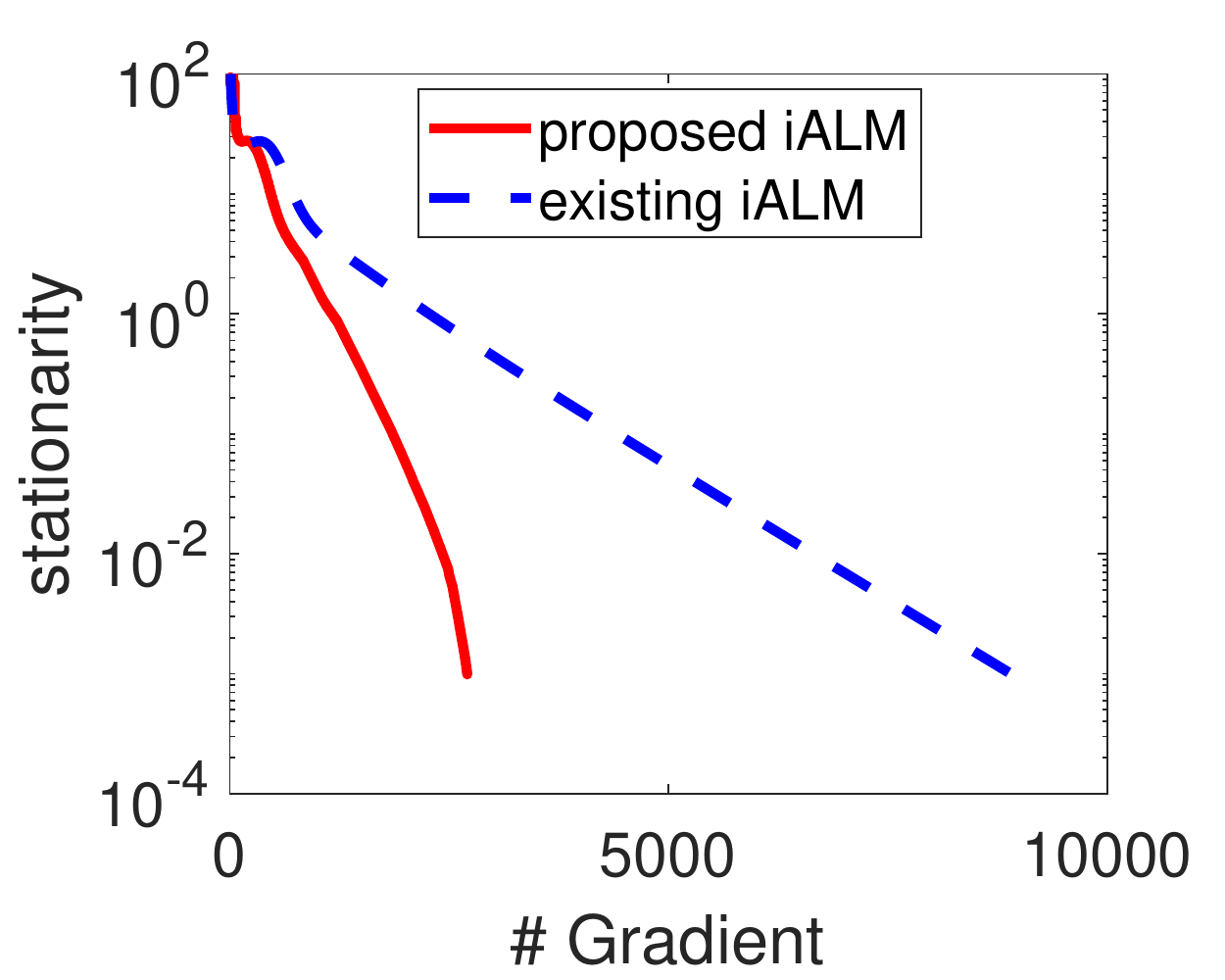} & \includegraphics[width=0.35\textwidth]{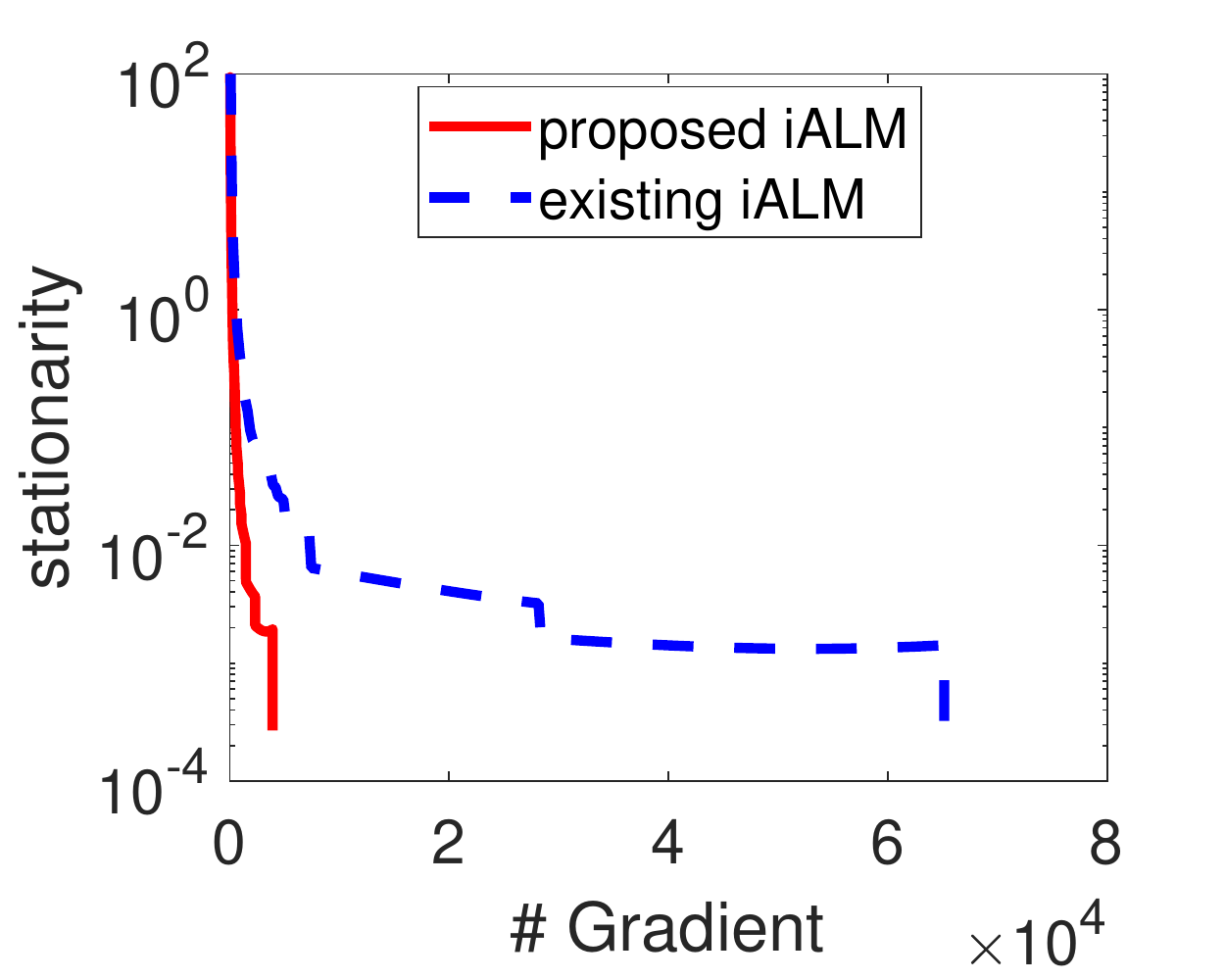} \\[0.3cm]
			\end{tabular}
		}
	\end{center}
	\vspace{-0.2cm}
	\caption{Comparison of iPPM and the subsolver of an existing iALM in \citep{sahin2019inexact} on solving the first subproblem of LCQP, EV, and clustering  problems. Each plot shows the violation of stationarity.}\label{fig:dres}
\end{figure}

\end{document}